%% file: main.tex
\documentclass[onecolumn,notitlepage,superscriptaddress]{revtex4-1}

\usepackage{graphicx}
\usepackage[utf8]{inputenc} % allow utf-8 input
\usepackage[T1]{fontenc}    % use 8-bit T1 fonts
\usepackage{hyperref}       % hyperlinks
\usepackage{url}            % simple URL typesetting
\usepackage{booktabs}       % professional-quality tables
\usepackage{amsfonts}       % blackboard math symbols
\usepackage{nicefrac}       % compact symbols for 1/2, etc.
\usepackage{microtype}      % microtypography
\usepackage{lipsum}
\usepackage{amsmath,amssymb,amsthm}
\usepackage[ruled]{algorithm2e}
\usepackage{mathtools}
\usepackage{physics}
\usepackage{subcaption}

\input{macro.tex}

\usepackage{tikz}
\usetikzlibrary{fit, arrows.meta, quotes}

\numberwithin{equation}{section}

\theoremstyle{definition}
\newtheorem{theorem}{Theorem}[section]

\newtheorem{proposition}{Proposition}[section]

\newtheorem{example}{Example}[section]
\newtheorem{assumption}{Assumption}[section]

\newtheorem{remark}{Remark}[section]

\begin{document}

\title{A fictitious-play finite-difference method\\ for linearly solvable mean field games}

\author{Daisuke Inoue}
\email{daisuke-inoue@mosk.tytlabs.co.jp}
\affiliation{%
Toyota Central R\&D Labs., Inc.\\
Bunkyo-ku, Tokyo 112-0004, Japan
}%
\author{Yuji Ito}
\affiliation{%
Toyota Central R\&D Labs., Inc.\\
Nagakute, Aichi 480-1192, Japan
}%
\author{Takahito Kashiwabara}
\affiliation{%
Graduate School of Mathematical Sciences, the University of Tokyo\\
3-8-1 Komaba, Meguro-ku, Tokyo 153-8914, Japan
}%
\author{Norikazu Saito}
\affiliation{%
Graduate School of Mathematical Sciences, the University of Tokyo\\
3-8-1 Komaba, Meguro-ku, Tokyo 153-8914, Japan
}%
\author{Hiroaki Yoshida}% <-this % stops a space
\affiliation{%
Toyota Central R\&D Labs., Inc.\\
Bunkyo-ku, Tokyo 112-0004, Japan
}%

\begin{abstract}
  \rev{An iterative finite difference scheme} for mean field games (MFGs) is proposed.
  The target MFGs are derived from control problems for multidimensional systems with advection terms.
  For such MFGs, linearization using the Cole-Hopf transformation and iterative computation using fictitious play are introduced.
  This leads to an implementation-friendly algorithm that iteratively solves explicit schemes.
  The convergence properties of the proposed scheme are mathematically proved by tracking the error of the variable through iterations.
  Numerical calculations show that the proposed method works stably for both one- and two-dimensional control problems.
\end{abstract}

% \keywords{mean field games; finite difference methods; fictitious play; Cole-Hopf transformation.}

\maketitle

\section{Introduction}

\rev{\emph{Mean field games (MFGs)} are models for approximating the Nash equilibrium for large-scale control problems in which an infinite number of homogeneous microscopic systems compete with each other~\cite{Lasry2007Mean}}.
This model is obtained by approximating the many-body interaction terms in each microscopic system as a two-body interaction between one system and the density distribution function formed by the other systems.
The resulting MFGs consist of two partial differential equations: a forward equation called the Fokker-Planck equation (FP equation), which describes the dynamics of the density distribution, and a backward equation called the Hamilton-Jacobi-Bellman equation (HJB equation), which describes the dynamics of the control input. 
The initial and terminal conditions are posed respectively for the FP and HJB equations.
The existence and uniqueness of solutions for MFGs have been proved under various assumptions~\cite{Cardaliaguet2010Notes, Gueant2011Mean}.

Several numerical methods for solving MFGs have been proposed. 
\rev{References~\cite{Achdou2012Mean, Achdou2010Meana, Achdou2012Iterative} propose methods for solving the discretized equations using Newton's method,}
and the convergence of the solution of the scheme to the solution of the MFGs has been proved.
To decouple the MFGs into HJB and FP equations, \rev{several} methods use a \rev{\emph{fixed-point iteration}} in which the FP and the HJB equations are solved alternately.
\rev{References~\cite{Camilli2012semidiscrete, Carlini2014Fully} propose such fixed-point methods in which each equation is solved with the semi-Lagrangian scheme.}
They consider only a quadratic Hamiltonian, and the convergence analysis \rev{through the fixed-point iteration} itself is not undertaken.
In Ref.~\cite{Gueant2012Mean}, a finite difference method with a standard \rev{fixed-point iteration} is proposed for the system obtained by applying the \emph{Cole-Hopf transformation} to the original MFGs.
The transformed system consists of two \emph{linear} equations, which overcome the difficulties stemming from the nonlinearity of the original MFGs.
The same paper also addresses the convergence analysis of \rev{fixed-point iteration}; however, only a quadratic Hamiltonian is considered.
An alternate decoupling method, called \emph{fictitious play}, is proposed in Ref.~\cite{Cardaliaguet2017Learninga}. 
There, the existence and uniqueness of a general class of MFGs are proved using fictitious play iteration. 
Furthermore, the solutions of the fictitious play are not merely approximations of the MFG equations; the solutions involve a learning process until arriving at equilibrium by using the developed histories. 
While Ref.~\cite{Cardaliaguet2017Learninga} does not address numerical calculations, 
\rev{fictitious play is also known to help obtain numerical convergence when fixed-point iterations fail to converge}~\cite{Lauriere2021Numerical, Perrin2020Fictitious}.

\begin{figure}[t]
  \centering
  \begin{tikzpicture}
    \node at (0,0) [rectangle,draw] (mfg) {MFG};
    \node at (5,0.5) [rectangle,draw] (imfg) {Iterative MFG};
    \node at (10,0.5) [rectangle,draw] (ilsmfg) {Iterative LSMFG};
    \node at (15,0) [rectangle,draw] (scheme) {Proposed Scheme};
    \draw [-{Latex[length=1.5mm]}] (mfg) to [bend left=10,"\footnotesize{\centering fictitious play}"] (imfg);
    \draw [-{Latex[length=1.5mm]}] (imfg) to [bend left=10,"\footnotesize{\centering Cole-Hopf transform}"] (ilsmfg);
    \draw [-{Latex[length=1.5mm]}] (ilsmfg) to [bend left=10,"\footnotesize{\centering discretize}"] (scheme);
    \draw [-{Latex[length=1.5mm]}] (imfg) to [bend left=10,"\footnotesize\parbox{3cm}{\centering Prop.~\ref{prop:PDE-existence-uniform-fp-convergence}: converge as $k\to\infty$}"] (mfg);
    \draw [-{Latex[length=1.5mm]}] (ilsmfg) to [bend left=10,"\footnotesize\parbox{3cm}{\centering  Prop.~\ref{prop:lsmfg_derivation}: equivalent}"] (imfg);
    \draw [-{Latex[length=1.5mm]}] (scheme) to [bend left=10,"\footnotesize\parbox{3cm}{Thm.~\ref{thm:convergence-repeat}: converge as $\Delta x, \Delta t\to 0$}"] (ilsmfg);
    \draw [-{Latex[length=1.5mm]}] (scheme.south) to [bend left=13,"\footnotesize{\centering Thm.~\ref{thm:convergence-repeat-limit}: converge as $\Delta x, \Delta t\to 0$, $k\to\infty$}"] (mfg.south);
  \end{tikzpicture}
  \caption{
    Conceptual diagram of convergence analysis.
  }
  \label{fig:conceptual-diagram}
\end{figure}
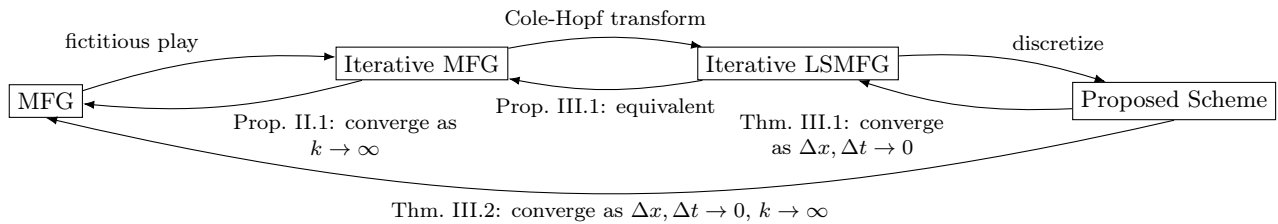

In the present paper, we propose a new numerical method for MFGs, using the same variable transformation technique as in Ref.~\cite{Gueant2012Mean}.
For this reason, we refer to the target MFGs as \emph{linearly solvable mean field games (LS-MFGs)}.
We discretize each equation with a finite difference scheme and adopt the fictitious play for the decoupling method. 
The contribution of our work can be summarized as follows.
\begin{enumerate}  
  \item A more general class of Hamiltonian than a quadratic $\rev{H(x,p)=\frac12 p^\top p}$ is addressed. 
  Specifically, our Hamiltonian is given as $\rev{H(x,p)=\frac12 p^\top B R^{-1}B^\top p -f(x)^\top p}$ (definitions of symbols will be described in the next section).  
  Actually, we are interested in the situation where the dynamics of the microscopic system include advection terms, which are terms containing $f(x)$ (see our target MFG \eqref{eq:LQHJB}--\eqref{eq:LQFP_initial}).
  Such advection terms are typically equipped by real-world control systems whose state is not directly controllable through the control input.
  \rerev{Thus, the problems involving such advection terms are recognized as important in the control community~\cite{nijmeijer1990nonlinear, Bagdasaryan2019OptimaL}.}
  However, in the previous studies \cite{Gueant2012Mean, Swiecicki2016Schr, Camilli2012semidiscrete,Carlini2014Fully}, the Hamiltonian is assumed to be $H(x,p)= \frac12 p^\top p$, which is more restrictive than our case.
  \rerev{References~\cite{Achdou2012Mean, Achdou2010Meana, Achdou2012Iterative} provide methods for solving MFGs with numerical Hamiltonians that satisfy appropriate conditions.
  Construction methods for such numerical Hamiltonians are given for some cases, e.g. $H(x,p) = \sup_a [p^\top a - L(x,a)]$ or $H(x,p) = \psi(x, |p|)$ for some function $\psi$.
  However, for our Hamiltonian with the advection term $f(x)$, the form of the numerical Hamiltonian is non-trivial.}
  For our target MFGs, we have shown that the Cole-Hopf transformation is applicable, whereby we find it possible to deal with the target MFGs as linearly solvable mean field games (LS-MFGs) (see Proposition \ref{prop:lsmfg_derivation}).

  \item 
  A finite difference scheme with fictitious play iteration is proposed for the multidimensional LS-MFG.
  Here, the derived LS-MFG is composed of two linear advection-diffusion-reaction equations.
  Since naive numerical calculations for this system lead to instability, we apply upwind differencing for the advection term, central differencing for the diffusion term, and \rev{implicit scheme for the reaction term. Such schemes are recently devised} to ensure that the solution of the scheme preserves positivity~\cite{Appadu2013Numericalc, Mickens1999Nonstandard, Mickens2000Nonstandard}.
  Consequently, our finite difference solutions satisfy the discrete maximum principle (see Propositions \ref{prop:bound_Psi} and \ref{prop:bound_tildePsi}). 
  This plays a crucial role in the convergence analysis described herein. 
  
  \item 
  The convergence properties of the proposed scheme are proved.
  More specifically, we have proved that the solution of the scheme converges to the solution of the MFG, \rev{in the limit where the discretization parameter is small and the number of iterations of fictitious play is large.
  This is carried out in two steps.
  The first step is showing the convergence of the discrete scheme to the fictitious play iterative system of LS-MFG (which we call the \emph{iterative LS-MFG}) in the limit of a small discretization parameter.
  The second step is to show the convergence of the iterative LS-MFG to the original MFG in the limit of a large number of fictitious play iterations.
  To perform the former step,}
  tracking the growth of the approximation errors of each equation through iterations is necessary.
  we thus start by proving the convergence properties of each forward and backward scheme (see Propositions \ref{thm:convergence} and \ref{thm:convergence-conservative}) and then perform a convergence analysis when the error circles back through the two equations
  \rev{(see Theorem \ref{thm:convergence-repeat}).
  The second step is accomplished by showing that the limit of the fictitious play solution converges to the MFG solution (see Proposition \ref{prop:PDE-existence-uniform-fp-convergence}) and that the solutions before and after the Cole-Hopf transformation correspond one-to-one (see Proposition \ref{prop:lsmfg_derivation}).
  The two steps together yield the final convergence theorem (see Theorem \ref{thm:convergence-repeat-limit}).
  A schematic diagram of the convergence analysis is shown in Fig.~\ref{fig:conceptual-diagram}.
  }
\end{enumerate}

In Section \ref{sec:MFG}, we briefly review the MFGs and their fictitious play.
We present our main results for the three novelties described above in Section \ref{sec:result}.
In Section \ref{sec:proof}, we give proof of the presented convergence theorem.
To confirm the validity of the proposed scheme, we perform numerical calculations for typical control problems for both one- and two-dimensional systems in Section \ref{sec:numeric}.
Finally, the conclusion is given in Section \ref{sec:conclusion}.

\textbf{Notation } 
The symbols $\partial_x$, $\partial_x\cdot$, and $\partial_{xx}$ denote the gradient operator, divergence operator, and Hessian operator, respectively.
The symbol $\Omega$ denotes a bounded closed set and $T>0$ denotes some positive real number.
For the matrices $A, B\in\bbR^{n\times n}$, the symbol $\Tr\{AB\}$ denotes $ \Tr\{AB\}= \sum_{i,j=1}^n A_{ij}B_{ji}$.
The set $\calC^{n}(\Omega)$ is the set of functions that are continuously differentiable for $n$-times.
\rev{The set $\calC^{n+\alpha}(\Omega)$ is the set of functions that are $\calC^n(\Omega)$ and their $n$-times derivatives are H\"{o}lder continuous of order $\alpha$.}
For any bounded continuous function $\phi:X\to\bbR$, we define $\|\phi\|_\infty \coloneqq\sup_{x\in X} |\phi(x)|$.
\rev{The set of bounded probability density functions defined on $\Omega$ is denoted as $\calP(\Omega)$ and the distance for $\mu,\nu\in\calP(\Omega)$ is denoted as
  \begin{align}
    \mathbf{d}_1(\mu, \nu)=\sup _h\left\{\int_{\Omega} h(x) [\mu(x) - \nu(x) ]\mathrm{d}x\right\},
  \end{align}
where the supremum is taken over all the maps $h: \Omega \to \bbR$ which are 1-Lipschitz continuous.
Let $U:\calP(\Omega)\to\bbR$ be a continuous map.
We say that the continuous map $\frac{\delta U}{\delta \mu}: \Omega\times\calP(\Omega)\to\bbR$ is the derivative of $U$ if, for any $\mu,\nu\in\calP(\Omega)$,
\begin{align}
  U\left(\nu\right)-U(\mu)=\int_0^1 \int_{\Omega} \frac{\delta U}{\delta \mu}\left(x, (1-s) \mu+s \nu\right)\left(\nu(x)-\mu(x)\right) \rmd x\mathrm{d} s.
\end{align}
}

\section{Review of mean field games and fictitious play}\label{sec:MFG}

Our target MFG equations \eqref{eq:LQHJB}--\eqref{eq:LQFP_initial} described below are not quite general, but are more general than those of Ref.~\cite{Gueant2012Mean}.
For the readers' convenience, we briefly recall a formal derivation of the MFG equations from the many-body control problems to clarify the motivation of this study.
Consider $N$ dynamical systems. 
The dynamics of the $i$th system is represented by the ordinary differential equation
\begin{align}\label{eq:general_dynamics}
  \rmd x_i(t) &= \qty{ f(x_i(t))+Ba_i(t)} \rmd t + \sigma \rmd w_i(t),\ t\in[0,T],\\
  x_i(0) &= x_i^0,\qquad i=1,\ldots,N,
\end{align}
where the variable $x_i (t) \in \bbR^n$ represents the state of the system at time $t\in [0,T]$ ($T>0$), and $a_i (t) \in \bbR^m$ denotes the \rev{deterministic} control input of the system at time $t$.
The variable $w_i(t)\in\bbR^n$ denotes the stochastic noise, which follows the standard Wiener process.
The function $f:\bbR^n\to \bbR^n$ characterizes the dynamics of the system, and the matrices $B\in\bbR^{n\times m}$ and $\sigma\in \bbR^{n\times n}$ are parameters of the system.

Next, we define an evaluation function for the input $a_i (t)$ from time $0$ to $T$ as
\begin{align}
  \begin{split}\label{eq:general_eval_func}
      J (\{a_i (t); a_{i^-} (t) \}_{t \in [0, T]} )&=  \rev{ \bbE \qty{\int_0^T \left[ \frac{1}{2}a_i(t)^\top R a_i(t) + g\left(x_i(t), m_{i^-}(\cdot,t)\right)\right] \rmd t + v_T(x_i(T)) }}.
    \end{split}
\end{align}
Here, the variable $a_{i ^-}(t)$ is a vector comprising the inputs for all systems except the $i$th system.
The function \rev{$m_{i^-}:\bbR^n\times[0,T] \to \bbR_{\ge0}$} is the empirical density function defined by
\begin{align}\label{eq:density}
  m_{i^-}(x,t) := \frac{1}{N-1}\sum_{j\ne i} \delta(x - x_j(t)),
\end{align}
where $\delta$ denotes the Dirac delta function.
The function \rev{$g:\bbR^n\times \calP(\bbR^n)\to\bbR$} represents the cost value calculated from the interaction of the $i$th system with other systems.
Furthermore, the function $v_T:\bbR^n\to\bbR$ represents the terminal cost.
The matrix $R\in\bbR^{m\times m}$ is positive definite and represents the penalty of the input.

We aim to find a combination of inputs $a_i^*$ that satisfies the following inequality for any $a_i\ \rev{(i=1,\ldots,N)}$:
\begin{align}\label{eq:nash_equilibrium}
  J\qty(\{a_i^*(t);a_{i^-}^*(t)\}_{t\in[0,T]}) \le  J\qty(\{a_i(t);a_{i^-}^*(t)\}_{t\in[0,T]}).
\end{align}
Such a set of inputs $a_i^* \ (i = 1,\ldots,N) $ is called a \emph{Nash equilibrium}.
\rev{Obtaining the Nash equilibrium requires solving a differential game in which Eqs.~\eqref{eq:general_dynamics} and \eqref{eq:general_eval_func} are coupled, which becomes increasingly difficult as the number of systems $N$ increases.}

On the other hand, when $N$ is infinitely large, $m_{i^-}(x,t)$ is approximated by a continuous density function $m(x,t)$.
Assuming that $m$ is known, the solution to the control problem is known to be characterized by the Hamilton-Jacobi-Bellman equation (HJB equation)~\cite{Fleming2006Controlled} taking the form
\begin{align}
  \begin{split}    
    - \partial_t v(x ,t) &= \partial_x v(x,t)^\top  f(x) + \rev{g(x,m(\cdot,t))}
     -\frac{1}{2} \partial_x v(x,t)^\top B R^{-1}B^\top \partial_x v(x,t) + \Tr \{\nu \partial_{xx}v(x,t)\},
  \end{split}\\
  v(x,T) &= v_T(x),
\end{align}
where the function $v:\bbR^n\times\bbR\to\bbR$ is called a \emph{value function}, and the variable $\nu$ is defined as $\nu\coloneqq \nicefrac{1}{2}\ \sigma^\top\sigma$.
Next, assuming that the control input $a^*$ is known and the density function at the initial time of the system is given by $m_0$, the time evolution of the density function is described by the Fokker-Planck equation (FP equation)~\cite{Gardiner2009Stochastic} as follows:
\begin{align}
  \partial_t m(x,t) &= - \partial_x \cdot\left[ \qty{ f(x) -B R^{-1}B^\top \partial_x v(x,t)  }m(x,t)\right] + \Tr \{\nu \partial_{xx}m(x,t)\},\\
  m(x,0)&=m_0(x).
\end{align}
Consequently, our target MFG equations are obtained as the coupling of the HJB and FP equations:
\begin{align}
  \begin{split}
    - \partial_t v(x ,t) &= \partial_x v(x,t)^\top  f(x) + \rev{g(x,m(\cdot,t))}\\
    &\quad -\frac{1}{2} \partial_x v(x,t)^\top B R^{-1}B^\top \partial_x v(x,t)\\
    &\quad + \Tr \{\nu \partial_{xx}v(x,t)\}
  \end{split}\qquad\qquad (x,t)\in\Omega\times [0,T],\label{eq:LQHJB}\\
  \begin{split}
    \partial_t m(x,t) &= - \partial_x \cdot\left[ \qty{ f(x) -B R^{-1}B^\top \partial_x v(x,t)  }m(x,t)\right]\\
    &\quad + \Tr \{\nu \partial_{xx}m(x,t)\}
  \end{split}\qquad (x,t)\in\Omega\times [0,T],\label{eq:LQFP}\\
  v(x,T) &= v_T(x)\qquad x\in\Omega,\label{eq:LQHJB_terminal}\\
  m(x,0)&=m_0(x)\qquad x\in\Omega,\label{eq:LQFP_initial}
\end{align}
where the set $\Omega$ is the spatial domain in $\mathbb{R}^n$.
We consider \eqref{eq:LQHJB}--\eqref{eq:LQFP_initial} under suitable boundary conditions, although we do not state them explicitly here.
Once a solution of \eqref{eq:LQHJB}--\eqref{eq:LQFP_initial} is deduced, the function
\begin{align}
  a^*(x,t) = -R^{-1}B^\top \partial_x v(x,t),
\end{align}
gives an approximate Nash equilibrium in \eqref{eq:nash_equilibrium} (see Ref.~\cite{Lasry2007Mean}).

\rev{We aim to construct a finite difference scheme for approximating} \eqref{eq:LQHJB}--\eqref{eq:LQFP_initial}.
Since the system is not an initial value problem, producing the solution, say the MFG equilibrium, presents a serious difficulty. 
The well-known approach is a \rev{fixed-point iteration} (see Ref.~\cite{Lauriere2021Numerical}, as an example).
In this approach, we generate alternately a sequence of solutions $v^{(k)}$ and $m^{(k)}$ of the HJB and FP equations and are able to solve the HJB and FP equations separately. 
Convergence properties are studied in \rev{Refs.~\cite{Carlini2014Fully,Gueant2012Mean}}, for example. 
However, some stability issues still remain (see Ref.~\cite{Lauriere2021Numerical}).   

In this paper, we study another iterative approach, called \emph{fictitious play}, as proposed in Ref.~\cite{Cardaliaguet2017Learninga}.
We apply the iterative scheme of fictitious play to the MFG equations \eqref{eq:LQHJB}--\eqref{eq:LQFP_initial}. We start with the initial guess $m^{(0)}=m^{(0)}(x,t)$. Then, for $k=0,1,\ldots$, we solve successively the following equations:
\begin{align}
  \begin{split}
    - \partial_t v^{(k+1)}(x ,t) &= \partial_x v^{(k+1)}(x,t)^\top  f(x) + \rev{g(x,\bar m^{(k)}(\cdot,t))} \\
    &\quad -\frac{1}{2} \partial_x v^{(k+1)}(x,t)^\top B R^{-1}B^\top \partial_x v^{(k+1)}(x,t)\\
    &\quad  + \Tr \{\nu \partial_{xx}v^{(k+1)}(x,t)\}
  \end{split}\qquad\qquad (x,t)\in\Omega\times [0,T],\label{eq:fp_LQHJB}\\
  \begin{split}    
    \partial_t m^{(k+1)}(x,t) &= - \partial_x \cdot\left[ \qty{ f(x) -B R^{-1}B^\top \partial_x v^{(k+1)}(x,t)  }m^{(k+1)}(x,t)\right]\\
    &\quad + \Tr \{\nu \partial_{xx}m^{(k+1)}(x,t)\}
  \end{split}\qquad(x,t)\in\Omega\times [0,T],\label{eq:fp_LQFP}\\
  v^{(k+1)}(x,T) &= v_T(x)\qquad x\in\Omega,\label{eq:fp_LQHJB_terminal}\\
  m^{(k+1)}(x,0)&=m_0(x)\qquad x\in\Omega,\label{eq:fp_LQFP_initial}
\end{align}
where $\bar m^{(k)}$ is defined as
\begin{align}\label{eq:m_bar_def}
  \bar m^{(k)}(x,t) &= \frac{1}{k+1}\sum_{\ell=0}^k m^{(\ell)}(x,t).
\end{align}

\begin{remark}
  Fictitious play helps obtain convergence when fixed-point iterations fail to converge (see Refs.~\cite{Lauriere2021Numerical, Perrin2020Fictitious}).
  Furthermore, the solutions of the fictitious play are not merely approximations of the MFG equations; they provide a learning process until arriving at equilibrium using the developed histories.
  Note that using the average of the entire history of the variable may slow down convergence.
  \rerev{
    For overcoming this, a generalization of fictitious play has been proposed that uses an average value of the variables over the past history with non-uniform weights, rather than an average value with uniform weights~\cite{Lavigne2022Generalized}. 
    In this paper, we focus on the standard fictitious play and conduct a convergence analysis of the algorithm.
    The analysis for the generalized algorithm is our future work.
  }
\end{remark}

The convergence properties of fictitious play are well studied in Ref.~\cite{Cardaliaguet2017Learninga} for a general class of MFGs. 
To apply the results, we formulate the MFG equations more precisely. 
First, we consider the equations in the $n$-dimensional cube $[-1,1]^n$ with the periodic boundary conditions. 
That is, we set 
\begin{equation}
 \Omega = \mathbb{R}^n/ (2\mathbb{Z})^n.
\end{equation}
Moreover, $T$ is a given positive constant. 
For the given functions $m_0$, $v_T$, $f$, and $g$, we make the following assumption.

\begin{assumption}\label{assump:PDE-existence-uniform-fp-convergence}
  \begin{enumerate}
    \item $m_0\in \calC^1(\Omega)\cap \calP(\Omega)$.
    \item $v_T\in \calC^3(\Omega)$.
    \item $f\in \calC^1(\Omega)$ and \rev{$B$ has full column rank}.
    \item 
    $m\mapsto g(\cdot,m)$ is Lipschitz continuous from $\calP(\Omega)$ to $\calC^2(\Omega)$:
    \begin{align}
      \|g(\cdot,m)-g(\cdot,m')\|_{\calC^2(\Omega)} \le L_g\bfd_1(m,m').
    \end{align}
    where $L_g$ is the Lipschitz constant.
    \item $g$ is monotone: 
    \begin{align}
      \int_{\Omega}\left(g(x, m)-g\left(x, m^{\prime}\right)\right) \left( m(x)-m^{\prime}(x)\right) \rmd x \geq 0.
    \end{align}
    In addition, there exists a function $G:\calP(\Omega)\to\bbR$ which satisfies
    \begin{align}
      \frac{\delta G}{\delta m} (x,m) = g(x,m).
    \end{align}      
    Finally, $g$ is uniformly bounded on $\Omega\times \calP(\Omega)$; we denote $\|g\|_{\infty} = \sup_{(x,m)\in \Omega\times\calP(\Omega)}|g(x,m)|< \infty$.
  \end{enumerate}
\end{assumption}

\begin{proposition}\label{prop:PDE-existence-uniform-fp-convergence}
  Suppose that Assumption \ref{assump:PDE-existence-uniform-fp-convergence} is satisfied.
  Then there exists a unique solution $\qty(v^{(k)}, m^{(k)})$ for \eqref{eq:fp_LQHJB}--\eqref{eq:fp_LQFP_initial} and a unique solution $(v, m)$ for \eqref{eq:LQHJB}--\eqref{eq:LQFP_initial} in the viscosity sense for the HJB equation and in the distribution sense for the FP equation.
  Moreover, those solutions are continuous in $\Omega\times [0,T]$ and
  \begin{equation}
    \lim_{k\to\infty}v^{(k)}= v,\quad 
    \lim_{k\to\infty}m^{(k)}= m,
  \end{equation}
  uniformly in $\Omega\times [0,T]$.
\end{proposition}
\begin{proof}
  This is a direct application of Theorem 2.1 in Ref.~\cite{Cardaliaguet2017Learninga}.
\end{proof}
\begin{remark}
  \rev{In \cite{Cardaliaguet2010Notes}, it is shown that the solution $(m, v)$ of the MFGs are of class $\calC^{2+\eta}$ for the first variable and class $\calC^{1+\eta/2}$ for the second variable, respectively, where $\eta\in(0,1/2)$.}
\end{remark}

\section{Main results}\label{sec:result}

In the preceding section, we mentioned our target MFG equations \eqref{eq:LQHJB}--\eqref{eq:LQFP_initial} and the associated fictitious play iterative scheme \eqref{eq:fp_LQHJB}--\eqref{eq:fp_LQFP_initial}.
In this section, we present the novel contributions of our work. 
Assuming further a relationship between the coefficients $R$, $B$, and $\nu$, and applying a variable transformation, we derive an alternate formulation of the MFG equations consisting solely of linear equations. 
Consequently, we can solve the MFG equations by solving only linear equations. 
For this reason, we call the resulting equations \emph{linearly solvable mean field games (LS-MFGs)} following Ref.~\cite{Todorov2007LinearlySolvable}.

\subsection{Derivation of LS-MFGs}\label{sec:LSMFG}

Hereinafter, we will suppose that the following assumption is satisfied. 
\begin{assumption}\label{assump:LHJB}
  There exists a positive constant $\lambda$ satisfying
  \begin{align}\label{eq:assump-LHJB}
    \lambda BR^{-1}B^\top=\nu.
  \end{align}
\end{assumption}

The assumption in \eqref{eq:assump-LHJB} implies that there is a proportional relationship between the noise and the input.
Such a property is known to hold for several real-world dynamical systems~\cite{Kappen2005Linear}.

Suppose we are given solution $(v,m)$ of the MFG equations \eqref{eq:LQHJB}--\eqref{eq:LQFP}.
We then introduce new variables 
$\psi$ and $\tilde{\psi}$ using the \emph{Cole-Hopf transformation} as
 \begin{align}
    \psi(x,t)&=\exp(-\frac{v(x,t)}{\lambda}),\label{eq:cole-hopf}\\
    \tilde\psi(x,t)&=\frac{m(x,t)}{\psi(x,t)}\label{eq:cole-hopf2}.
  \end{align}
Under Assumption \ref{assump:LHJB}, the third and fourth terms on the right-hand side of \eqref{eq:LQHJB} are calculated as 
  \begin{align}
    \begin{split}      
      &\quad -\frac{1}{2}  \partial_x v^\top B R^{-1}B^\top \partial_x v + \Tr \{\nu \partial_{xx}v\}\\
      &= 
      -\frac{\lambda^2}{2\psi^2} \partial_x\psi^\top BR^{-1}B^\top\partial_x\psi 
      + \frac{\lambda}{2\psi^2} \partial_x\psi^\top\nu\partial_x\psi - \frac{\lambda}{2\psi}\Tr\{\nu \partial_{xx}\psi\}\\
      & = - \frac{\lambda}{2\psi}\Tr\{\nu \partial_{xx}\psi\}.
    \end{split}\label{eq:L-HJB_1}
  \end{align}
  from which we obtain \eqref{eq:ADR-non-conservative}.
A similar calculation for $\tilde{\psi}$ is also possible. 

Consequently, we find that $(\psi,\tilde{\psi})$ solves the following equations: 
\begin{align}
    \begin{split}\label{eq:ADR-non-conservative}
      \partial_t \psi(x ,t) &= -  f(x)^\top\partial_x \psi(x,t) - \Tr \{\nu \partial_{xx}\psi(x,t)\} + \frac{1}{\lambda}\rev{g(x,m(\cdot,t))}\psi(x,t)
    \end{split}\qquad (x,t)\in\Omega\times [0,T],\\
    \begin{split}\label{eq:ADR-conservative}
      \partial_t \tilde \psi(x ,t) &= - \partial_x\cdot( f(x) \tilde \psi(x,t)) + \Tr \{\nu \partial_{xx}\tilde \psi(x,t)\} - \frac{1}{\lambda}\rev{g(x,m(\cdot,t))} \tilde \psi(x,t)
    \end{split}\qquad (x,t)\in\Omega\times [0,T], \\
    \psi(x,T) &= \exp(-\frac{v_T(x)}{\lambda}) \qquad x\in\Omega,\\
    \tilde\psi(x,0) &= \frac{m_0(x)}{\psi(x,0)}\qquad x\in\Omega.\label{eq:ADR-non-conservative_initial}
  \end{align}
  where $m$ is given as
  \begin{align}
    m(x,t) = \psi(x,t)\tilde \psi(x,t).
  \end{align}
The converse is also true. That is, we have the following proposition. 

\begin{proposition}\label{prop:lsmfg_derivation}
  Suppose that Assumption \ref{assump:LHJB} holds. 
  If a smooth solution $(v,m)$ of the MFG equations \eqref{eq:LQHJB}--\eqref{eq:LQFP_initial} is given, then $(\psi,\tilde{\psi})$ defined by \eqref{eq:cole-hopf} and \eqref{eq:cole-hopf2} solves the system \eqref{eq:ADR-non-conservative}--\eqref{eq:ADR-non-conservative_initial}. 
  Conversely, if a smooth solution $(\psi,\tilde\psi)$ of \eqref{eq:ADR-non-conservative}--\eqref{eq:ADR-non-conservative_initial} is given, then $(v,m)\coloneqq(-\lambda \ln \psi, \psi\tilde\psi)$ solves the MFG equations \eqref{eq:LQHJB}--\eqref{eq:LQFP_initial}. 
\end{proposition}
We call the transformed system \eqref{eq:ADR-non-conservative}--\eqref{eq:ADR-non-conservative_initial} the \emph{linearly solvable mean field game (LS-MFG)} equations.

\begin{remark}
  The Cole-Hopf transformation has a long history.
  It is applied to \rev{convert HJB equations to linear equations}
  in Refs.~\cite{Kappen2005Linear, Kappen2005Path}.
  The LS-MFG equations for the simpler MFG equations are derived in Refs.~\cite{Gueant2012Mean, Swiecicki2016Schr}; however, these studies do not consider the case in which function $f(x)$ exists in the dynamics \eqref{eq:general_dynamics}.
  Equations~\eqref{eq:ADR-non-conservative}--\eqref{eq:ADR-non-conservative_initial} are extensions of the systems in Refs.~\cite{Gueant2012Mean, Swiecicki2016Schr}, which better describe real-world control problems.
  \rev{Note that the result of the Cole-Hopf transformation has an adjoint structure, which allows discretization with the same structure for the two equations.}
\end{remark}

\begin{proposition}\label{prop:CH-backward-bound}
  \rerev{Suppose that Assumption \ref{assump:PDE-existence-uniform-fp-convergence}-(1)--(4) and Assumption \ref{assump:LHJB} are satisfied.}
  There exist constants $\gamma_{\min}>0$ and $\gamma_{\max}>0$ such that for any $m\in\calP(\Omega)$, the solution $\psi$ for
  \begin{align}
    \begin{split}\label{eq:CH-backward}
      \partial_t \psi(x ,t) &= -  f(x)^\top\partial_x \psi(x,t) - \Tr \{\nu \partial_{xx}\psi(x,t)\} + \frac{1}{\lambda}\rev{g(x,m(\cdot,t))}\psi(x,t)
    \end{split}\qquad (x,t)\in\Omega\times [0,T],\\
    \psi(x,T) &= \exp(-\frac{v_T(x)}{\lambda}) \qquad x\in\Omega,\label{eq:CH-backward-0}
  \end{align}
  satisfies
  \begin{align}\label{eq:CH-backward-bound}
    0<\gamma_{\min} &\le \psi(x,t) \le\gamma_{\max}\quad \forall (x,t)\in\Omega\times[0,T].
  \end{align}
\end{proposition}
\begin{proof}
  See Appendix.    
\end{proof}

By performing the same transformation on \eqref{eq:fp_LQHJB}--\eqref{eq:fp_LQFP_initial}, the fictitious play for the LS-MFG is derived as
\begin{align}
  \begin{split}\label{eq:ADR-fp-non-conservative}
    \partial_t \psi^{(k+1)}(x ,t) &= -  f(x)^\top\partial_x \psi^{(k+1)}(x,t) - \Tr \{\nu \partial_{xx}\psi^{(k+1)}(x,t)\} \\
    &\quad + \frac{1}{\lambda}\rev{g(x,\bar m^{(k)}(\cdot,t))}\psi^{(k+1)}(x,t)
  \end{split}\qquad\quad (x,t)\in\Omega\times [0,T],\\
  \begin{split}\label{eq:ADR-fp-conservative}
    \partial_t \tilde \psi^{(k+1)}(x ,t) &= - \partial_x\cdot( f(x) \tilde \psi^{(k+1)}(x,t)) + \Tr \{\nu \partial_{xx}\tilde \psi^{(k+1)}(x,t) \}\\
    &\quad - \frac{1}{\lambda}\rev{g(x,\bar m^{(k)}(\cdot,t))} \tilde \psi^{(k+1)}(x,t)
  \end{split}\qquad\quad (x,t)\in\Omega\times [0,T],\\
  \psi^{(k)}(x,T) &= \exp(-\frac{v_T(x)}{\lambda})\qquad x\in\Omega,\label{eq:ADR-fp-non-conservative_terminal}\\
  \tilde\psi^{(k)}(x,0) &= \frac{m_0(x)}{\psi^{(k)}(x,0)}\qquad x\in\Omega,\label{eq:ADR-fp-conservative_initial}
\end{align}
where $\bar m^{(k)}(x,t)$ denotes
\begin{align}
  \bar m^{(k)}(x,t) &= \frac{1}{k+1}\sum_{\ell=0}^k m^{(\ell)}(x,t),\\
  m^{(k)}(x,t) &= \psi^{(k)}(x,t)\tilde \psi^{(k)}(x,t).\label{eq:def_m^{(k)}}
\end{align}

\begin{remark}
  \rev{
    The result obtained in Propositions \ref{prop:lsmfg_derivation} and \ref{prop:CH-backward-bound} are valid for the iterative MFGs.
    First, for a fixed $k\in\bbN$, the solution $(v^{(k)},m^{(k)})$ for \eqref{eq:fp_LQHJB}--\eqref{eq:m_bar_def} and the solution $(\psi^{(k)}, \tilde\psi^{(k)})$ for \eqref{eq:ADR-fp-non-conservative}--\eqref{eq:def_m^{(k)}} satisfy $(v^{(k)},m^{(k)})=(-\lambda \ln \psi^{(k)}, \psi^{(k)}\tilde\psi^{(k)})$.
    In addition, the boundedness for $\psi^{(k)}$ is given as
    \begin{align}\label{eq:psi_bound}
      0<\gamma_{\min} &\le \psi^{(k)}(x,t) \le\gamma_{\max} \quad \forall (x,t)\in\Omega\times[0,T],
    \end{align}
    which is obtained by the same proof as in Proposition \ref{prop:CH-backward-bound}.
  }
\end{remark}

\subsection{Proposed numerical scheme}\label{sec:scheme}

We propose a numerical scheme for the LS-MFG equations \eqref{eq:ADR-fp-non-conservative}--\eqref{eq:def_m^{(k)}}.
Recall that we are working on $\Omega = \bbR^n/(2\bbZ)^n$.
We start by discretizing the space variable $x\in \Omega$.
For each coordinate $x_l$, where $x_l\ (l=1,\ldots,n)$ denotes the $l$th component of $x$, we define a partition number $N_{x_l}>0$ and $\Delta x_l \coloneqq 1/N_{x_l}$, from which we define grid points $x_{i_l}\coloneqq i_l \Delta x_l\ (i_l = 0, \pm1, \ldots, \pm N_{x_l})$.
For the time variable $t\in[0,T]$, we define a partition number $N_t>0$ and $\Delta t\coloneqq T/N_t$, from which we define grid points $t_j \coloneqq j\Delta t\ (j=0,\ldots,N_t)$.
Unless otherwise noted, we use the notation $i_l \equiv 0, \pm1, \ldots, \pm N_{x_l}\ (\text{mod}\  2N_{x_l} + 1)$ and $j\equiv 0,\ldots, N_t$.
We vectorize $i_l$ and $x_{i_l}$ as $\bfi = (i_1,i_2,\ldots, i_n)$ and $x_\bfi = [x_{i_1}, x_{i_2}, \ldots, x_{i_n}]^\top$.

For each grid point $(\bfi,j)$, the approximate values of $\psi(x_\bfi,t_j)$, $\tilde\psi(x_\bfi,t_j)$, and $m(x_\bfi,t_j)$ are denoted by $\Psi_{\bfi,j}$, $\tilde\Psi_{\bfi,j}$, and $M_{\bfi,j}$, respectively.
The $l$th component of $f(x_\bfi)$ is denoted by $f_{l}(x_\bfi)$.
The difference quotients of the variables are defined as follows:
\begin{align}
  \begin{split}
    D(\Psi, \bfi,j) &= \sum_{l=1}^n\left( \frac{1+\sgn f_{l}(x_\bfi)}{2}f_{l}(x_\bfi) \frac{\Psi_{\bfi[{l^+}],j}-\Psi_{\bfi,j}}{\Delta x_l}\right.\\
     &\qquad\qquad + \left.\frac{1-\sgn f_{l}(x_\bfi)}{2}f_{l}(x_\bfi) \frac{\Psi_{\bfi,j}-\Psi_{\bfi[{l^-}],j}}{\Delta x_l}\right),
  \end{split}\\
  \begin{split}
    \tilD(\tilde\Psi, \bfi,j) &=\sum_{l=1}^n\left( -\frac{1+\sgn f_{l}(x_\bfi)}{2} \frac{f_{l}(x_\bfi)\tilde\Psi_{\bfi,j} - f_{l}(x_{\bfi[{l^-}]})\tilde\Psi_{\bfi[{l^-}],j}}{\Delta x_l}\right.\\
     &\qquad\qquad- \left.\frac{1-\sgn f_{l}(x_\bfi)}{2}  \frac{f_{l}(x_{\bfi[{l^+}]})\tilde\Psi_{\bfi[{l^+}],j} - f_{l}(x_\bfi)\tilde\Psi_{\bfi,j}}{\Delta x_l}\right),
  \end{split}\\
% \end{align}
% \begin{align}
  D_2(\Psi,\bfi,j) &=\sum_{l_1, l_2=1}^n\left( \nu_{l_1,l_2} \frac{\Psi_{\bfi[{l_1^+l_2^+}],j} - \Psi_{\bfi[{l_1^+l_2^-}],j} - \Psi_{\bfi[{l_1^-l_2^+}],j} + \Psi_{\bfi[{l_1^-l_2^-}],j}}{\Delta x_{l_1}\Delta x_{l_2}}\right),\\
  \tilD_2(\tilde\Psi,\bfi,j) &=\sum_{l_1, l_2=1}^n\left( \nu_{l_1,l_2} \frac{\tilde\Psi_{\bfi[{l_1^+l_2^+}],j} - \tilde\Psi_{\bfi[{l_1^+l_2^-}],j} - \tilde\Psi_{\bfi[{l_1^-l_2^+}],j} + \tilde\Psi_{\bfi[{l_1^-l_2^-}],j}}{\Delta x_{l_1}\Delta x_{l_2}}\right),
\end{align}
where we have set 
\begin{align}
  \bfi[{l^+}] &\coloneqq \left(i_{1}, i_{2}, \cdots, i_{l-1}, i_{l}+1, \cdots, i_{n}\right),\\
  \bfi[{l^-}] &\coloneqq \left(i_{1}, i_{2}, \cdots, i_{l-1}, i_{l}-1, \cdots, i_{n}\right),
\end{align}
and
\begin{align}
  \bfi[{l_1^+l_2^+}] &\coloneqq 
  \begin{cases}
    \left(i_{1}, i_{2}, \cdots, i_{l_1-1}, i_{l_1}+1, \cdots, i_{l_2-1}, i_{l_2}+1,\cdots,i_{n}\right),& l_1\ne l_2,\\
    \left(i_{1}, i_{2}, \cdots, i_{l_1-1}, i_{l_1}+1, \cdots,i_{n}\right), & l_1= l_2,
  \end{cases}\\
  \bfi[{l_1^-l_2^-}] &\coloneqq 
  \begin{cases}
    \left(i_{1}, i_{2}, \cdots, i_{l_1-1}, i_{l_1}-1, \cdots, i_{l_2-1}, i_{l_2}-1,\cdots,i_{n}\right),& l_1\ne l_2,\\
    \left(i_{1}, i_{2}, \cdots, i_{l_1-1}, i_{l_1}-1, \cdots,i_{n}\right), & l_1= l_2,
  \end{cases}\\  
  \bfi[{l_1^+l_2^-}] &\coloneqq 
  \begin{cases}
    \left(i_{1}, i_{2}, \cdots, i_{l_1-1}, i_{l_1}+1, \cdots, i_{l_2-1}, i_{l_2}-1,\cdots,i_{n}\right),& l_1\ne l_2,\\
    \left(i_{1}, i_{2}, \cdots,i_{n}\right)=\bfi, & l_1= l_2,
  \end{cases}\\
  \bfi[{l_1^-l_2^+}] &\coloneqq 
  \begin{cases}
    \left(i_{1}, i_{2}, \cdots, i_{l_1-1}, i_{l_1}-1, \cdots, i_{l_2-1}, i_{l_2}+1,\cdots,i_{n}\right),& l_1\ne l_2,\\
    \left(i_{1}, i_{2}, \cdots,i_{n}\right)=\bfi, & l_1= l_2.
  \end{cases}
\end{align}
Using the notation above, we define the finite difference scheme for \eqref{eq:ADR-fp-non-conservative}--\eqref{eq:ADR-fp-conservative_initial} as
\begin{align}
  \begin{split}\label{eq:nd_upwind_ADR}  
    \frac{\Psi^{(k+1)}_{\bfi,j-1}-\Psi^{(k+1)}_{\bfi,j}}{\Delta t} &=  D(\Psi^{(k+1)},\bfi,j) + D_2(\Psi^{(k+1)},\bfi,j)
     -\frac{1}{\lambda} \rev{g\left( x_\bfi, \bar M_{j}^{(k)}\right)}\Psi^{(k+1)}_{\bfi,j-1}, 
  \end{split}\\
  \begin{split}\label{eq:nd_upwind_ADR_conservative}
    \frac{\tilde\Psi^{(k+1)}_{\bfi,j+1}-\tilde\Psi^{(k+1)}_{\bfi,j}}{\Delta t} &= \tilD(\tilde\Psi^{(k+1)},\bfi,j) + \tilD_2(\tilde\Psi^{(k+1)},\bfi,j)
     - \frac{1}{\lambda} \rev{g\left( x_\bfi, \bar M_{j}^{(k)}\right)}\tilde\Psi^{(k+1)}_{\bfi,j+1}, 
  \end{split}\\
  \Psi_{\bfi,N_t}^{(k)} &= \exp(-\frac{v_T(x_\bfi)}{\lambda}),\label{eq:nd_upwind_ADR_terminal}\\
  \tilde\Psi_{\bfi,0}^{(k)} &= \frac{m_0(x_\bfi)}{\Psi_{\bfi,0}^{(k)}},\label{eq:nd_upwind_ADR_conservative_initial}
\end{align}
where we have defined 
\begin{align}
  M_{\bfi,j}^{(k)} &= \tilde \Psi_{\bfi,j}^{(k)}\Psi_{\bfi,j}^{(k)},\label{eq:nd_M}\\
  \rev{M_{j}^{(k)}(x)} &\rev{= M_{\bfi,j}^{(k)} \quad \text{for } x \in [x_{\bfi-\bfone \Delta x/2}, x_{\bfi+\bfone \Delta x/2}) },\\
  \rev{\bar M_{j}^{(k)}(x)} &\rev{= \frac{1}{k+1}\sum_{\ell=0}^k M_{j}^{(k)}(x) }\label{eq:nd_cumurative_M}.
\end{align}

Using the scheme above, we obtain the solution with \rev{Algorithm 1}, where the forward and backward equations are solved repeatedly using fictitious play.
\rev{Here, $M^{(0)}$ must be chosen at the first iteration. We set $M_{\bfi,j}^{(0)}=m^{(0)}(x_\bfi, t_j)$ for all $\bfi$ and $j$.
For the stopping criteria $\epsilon$ for the algorithm, we choose the error of the variables from the previous iteration $\| M^{(k)} - M^{(k-1)} \|_\infty$. 
Note that the choice of the stopping criteria is not unique; for instance, the residuals of the two finite difference equations could also be used. 
We numerically examined the impact of the choice of the criteria, to confirm that it does not affect the results significantly.
}

\begin{algorithm}[t]
  % \SetAlgoLined
  \caption{Numerical Method for the LS-MFG}\label{alg:LSMFG}
  \SetKwInOut{Input}{Input}
  \SetKwInOut{Output}{Output}
  \Input{$\epsilon>0$: Parameter for convergence assessment}
  \Output{$\Psi$, $\tilde\Psi$: Solution for the LSMFG}
  $k\leftarrow 0$, $z\leftarrow +\infty$, $\rev{M_{\bfi,j}^{(0)}\leftarrow m^{(0)}(x_\bfi, t_j)\ \forall \bfi,j}$\\
  \While{$z>\epsilon$}{
    Calculate $\bar M^{(k)}$ with \eqref{eq:nd_cumurative_M}, using $M^{(\ell)}$ for $\ell=0,\ldots, k$.\\
    Solve \eqref{eq:nd_upwind_ADR} backwardly with $\bar M^{(k)}$ and $v_T$, and obtain $\Psi^{(k+1)}$.\\
    Calculate $\tilde\Psi_{\cdot,0}^{(k+1)}$ with \eqref{eq:nd_upwind_ADR_conservative_initial}, using $\Psi_{\cdot,0}^{(k+1)}$ and $m_0$.\\
    Solve \eqref{eq:nd_upwind_ADR_conservative} forwardly with $\bar M^{(k)}$ and $\tilde\Psi_{\cdot,0}^{(k+1)}$, and obtain $\tilde\Psi^{(k+1)}$.\\
    Calculate $M^{(k+1)}$ with \eqref{eq:nd_M}, using $\tilde \Psi^{(k+1)}$ and $\Psi^{(k+1)}$.\\
    $z\leftarrow \| M^{(k)} - M^{(k-1)} \|_\infty$.\\
    $k\leftarrow k+1$.
  }
  $\Psi\leftarrow \Psi^{(k)}$, $\tilde\Psi\leftarrow \tilde\Psi^{(k)}$. 
\end{algorithm}

\begin{remark}
  In Ref.~\cite{Gueant2012Mean}, an implicit scheme is proposed only for one-dimensional systems, whereas we here propose an explicit scheme for multidimensional systems.
  The scheme above uses upwind differencing for the advection term, central differencing for the diffusion term,
  \rev{and implicit scheme for the reaction term.
  Such a discretization is originally proposed in \cite{Mickens1999Nonstandard} to guarantee the positivity of the solution for the reaction-diffusion equation.
  See also \cite{Mickens2000Nonstandard, Appadu2013Numericalc} for details.
  }
  However, it has not yet been applied in the advection-diffusion-reaction equation, nor has the convergence analysis of the schemes been done.
  We prove the convergence properties below.
\end{remark}

\subsection{Convergence properties}\label{sec:converge_result}

We proceed to our convergence analysis.
To avoid unimportant difficulties, we restrict our attention to the case $n=1$, that is, $\Omega = \bbR/(2\bbZ)$.
The proof for the general $n$-dimensional problem is also possible in essentially the same manner.

For a one-dimensional problem,
we discretize the space variable $x$ as $x_i = i\Delta x\ (i=-N_x,\ldots,N_x)$, where we have defined $\Delta x\coloneqq 1/N_x$ with partition numbers $N_x>0$.
We define $f_i \coloneqq  f(x_i)$ and $f_i^{+} \coloneqq \max(f_i, 0)$, $f_i^{-} \coloneqq \min(f_i, 0)$.
Then the difference scheme in \eqref{eq:nd_upwind_ADR}--\eqref{eq:nd_upwind_ADR_conservative_initial} can be rewritten as
\begin{align}
  \begin{split}    
    \frac{\Psi^{(k+1)}_{i,j-1}-\Psi^{(k+1)}_{i,j}}{\Delta t} &= f_i^{+} \frac{\Psi^{(k+1)}_{i+1,j}-\Psi^{(k+1)}_{i,j}}{\Delta x}+f_i^{-} \frac{\Psi^{(k+1)}_{i,j}-\Psi^{(k+1)}_{i-1,j}}{\Delta x}\\
    &\quad + \nu\frac{\Psi^{(k+1)}_{i+1,j}-2\Psi^{(k+1)}_{i,j}+\Psi^{(k+1)}_{i+1,j}}{\Delta x^2}
     - \frac{1}{\lambda} \rev{g\left(x_i, \bar M_{j}^{(k)}\right)}\Psi^{(k+1)}_{i,j-1}, 
  \end{split}\label{eq:upwind_ADR}\\
  \begin{split}    
    \frac{\tilde\Psi^{(k+1)}_{i,j+1}-\tilde\Psi^{(k+1)}_{i,j}}{\Delta t} &= -\frac{1+\sgn f_i}{2} \frac{f_{i}\tilde\Psi^{(k+1)}_{i,j} - f_{i-1}\tilde\Psi^{(k+1)}_{i-1,j}}{\Delta x}
     - \frac{1-\sgn f_i}{2}  \frac{f_{i+1}\tilde\Psi^{(k+1)}_{i+1,j} - f_{i}\tilde\Psi^{(k+1)}_{i,j}}{\Delta x}\\
    &\quad + \nu\frac{\tilde\Psi^{(k+1)}_{i+1,j}-2\tilde\Psi^{(k+1)}_{i,j}+\tilde\Psi^{(k+1)}_{i+1,j}}{\Delta x^2}
     - \frac{1}{\lambda} \rev{g\left(x_i, \bar M_{j}^{(k)}\right)}\tilde\Psi^{(k+1)}_{i,j+1}, 
  \end{split}\label{eq:upwind_ADR_conservative}
\end{align}
\begin{align}
  \Psi_{i,N_t}^{(k)} &= \exp(-\frac{v_T(x_i)}{\lambda}),\label{eq:upwind_ADR_terminal}\\
  \tilde\Psi_{i,0}^{(k)} &= \frac{m_0(x_i)}{\Psi_{i,0}^{(k)}}.\label{eq:upwind_ADR_conservative_initial}
\end{align}
To conduct the convergence analysis, we assume the following.
\begin{assumption}\label{assmp:convergence_analysis}
  \begin{enumerate}
    \item For \eqref{eq:ADR-fp-non-conservative}--\eqref{eq:ADR-fp-conservative_initial}, suppose that the solutions $\psi^{(k)}(x,t)$ and $\tilde\psi^{(k)}(x,t)$ are of class \rev{$\calC^{2+\eta}$} for the first variable and class $\rev{\calC^{1+\frac{\eta}{2}}}$ for the second variable, respectively, \rev{where $\eta\in(0,1]$}.
    \item The integers $N_x$ and $N_t$ are sufficiently large so that $\Delta x\le 1$ and $\Delta t \le 1$.
    In addition, the Courant-Friedrichs-Lewy condition (CFL condition) is satisfied:
    \begin{align}\label{eq:CFL}
      \alpha\sup_{x\in\Omega} |f(x)| + 2\beta \nu  < 1,
    \end{align}
    where we have defined 
    \begin{align}
      \alpha&\coloneqq\frac{\Delta t}{\Delta x},\\
      \beta&\coloneqq \frac{\Delta t}{\Delta x^2}.
    \end{align}
  \end{enumerate}
\end{assumption}
The main convergence results are provided below.
\begin{theorem}\label{thm:convergence-repeat}
  Let $(\psi^{(k)}, \tilde\psi^{(k)})$ be solutions of the fictitious play for the LS-MFG in \eqref{eq:ADR-fp-non-conservative}--\eqref{eq:def_m^{(k)}} and $(\Psi^{(k)}, \tilde\Psi^{(k)})$ be solutions of the finite difference scheme in \eqref{eq:upwind_ADR}--\eqref{eq:upwind_ADR_conservative_initial}.  
  \rerev{Suppose that Assumption \ref{assump:PDE-existence-uniform-fp-convergence}-(1)--(4), \ref{assump:LHJB}, \ref{assmp:convergence_analysis} are satisfied.}
  Then, there exist strictly increasing positive functions $K(k)$ and $\tilK(k)$ of $k$ such that
  \begin{align}
    &\sup_{i,j} \left|\psi^{(k)}(x_i,t_j)-\Psi_{i,j}^{(k)}\right| \le K(k)\rev{\qty((\Delta x)^\eta + (\Delta t)^\frac{\eta}{2})}\label{eq:psi_convergence},\\
    &\sup_{i,j} \left|\tilde \psi^{(k)}(x_i,t_j)-\tilde \Psi_{i,j}^{(k)}\right| \le \tilde K(k)\rev{\qty((\Delta x)^\eta + (\Delta t)^\frac{\eta}{2})}\label{eq:tilpsi_convergence}.
  \end{align}
\end{theorem}

\begin{theorem}\label{thm:convergence-repeat-limit}
  Let $(m, v)$ be solutions of the MFG in \eqref{eq:LQHJB}--\eqref{eq:LQFP_initial} and $(\Psi^{(k)}, \tilde\Psi^{(k)})$ be solutions of the finite difference scheme in \eqref{eq:upwind_ADR}--\eqref{eq:upwind_ADR_conservative_initial}.
  Suppose that Assumptions \ref{assump:PDE-existence-uniform-fp-convergence}, \ref{assump:LHJB}, \ref{assmp:convergence_analysis} are satisfied.
  Then, the following convergence properties hold:
  \begin{align}
    \lim_{k\to\infty}
    \lim_{\rev{\Delta t, \Delta x\to 0}}
    \sup_{i,j} \left|m(x_i,t_j)-\Psi_{i,j}^{(k)}\tilde\Psi_{i,j}^{(k)}\right| =0,\label{eq:m_convergence_repeat}\\
    \lim_{k\to\infty}
    \lim_{\rev{\Delta t, \Delta x\to 0}}
    \sup_{i,j} \left|v(x_i,t_j)-\qty(-\lambda \ln \Psi_{i,j}^{(k)}) \right| =0.\label{eq:v_convergence_repeat}
  \end{align}
\end{theorem}

\begin{remark}
  The theorems above describe the behavior of the error between the solution of the scheme and the solution of the PDE when the fictitious play iterations are repeated. 
  Theorem \ref{thm:convergence-repeat} describes the magnitude of the error with respect to a finite number of iterations $k$ and discretization parameters $\Delta x$ and $\Delta t$, while Theorem \ref{thm:convergence-repeat-limit} describes the limiting properties (see the proof in Section \ref{sec:proof} for the precise meaning of limit operation).
  In the process of proving these theorems, the convergence of the scheme for the respective forward and backward equations needs to be stated. 
  This is addressed in Propositions \ref{thm:convergence} and \ref{thm:convergence-conservative} in Section \ref{sec:proof}.
  \rev{
  Note that the convergence may not be uniform with respect to $k$ in Theorem \ref{thm:convergence-repeat}. 
  In addition, the convergence of $\Psi^{(k)}$ and $\tilde \Psi^{(k)}$ in Theorem \ref{thm:convergence-repeat-limit} for \rerev{fixed discretization parameters} has not been addressed and is left as an issue for future work.
  }
\end{remark}

\begin{remark}
  While the above theorems are only for one-dimensional systems, the proofs in Section \ref{sec:proof} do not contain operations that are impossible in two or more dimensions.
  We thus believe that it is possible to give similar results for $n$-dimensional systems by appropriately extending CFL condition \eqref{eq:CFL}.
  However, since this would be lengthy, the results we give here are limited to the one-dimensional case.
\end{remark}

\section{Proof of theorems}\label{sec:proof}

Before starting the convergence proof, we give a bound for distances on $\calP(\Omega)$.
\begin{proposition}\label{prop:d1-sup}
  \rev{There exists a positive constant $C_\Omega$, and for any two probability density functions $m, m'\in\calP(\Omega)$, the following is satisfied:}
  \begin{align}
    \rev{\bfd_1(m,m')\le C_\Omega\|m-m'\|_\infty = C_\Omega \sup_{x\in\Omega}|m(x)-m'(x)|.  }
  \end{align}
\end{proposition}
\begin{proof}
  \rev{See Appendix.}
\end{proof}
\rerev{For the proof, we also give a proposition that states the independence of variables.}
\begin{proposition}
  \rerev{
    Suppose that Assumption \ref{assmp:convergence_analysis}-(1) holds.
    Let $L_\psi^{(k)}$, $L_{\tilde\psi}^{(k)}$, $L_{\psi_{xx}}^{(k)}$, $L_{\psi_{t}}^{(k)}$, $L_{\tilde\psi_{xx}}^{(k)}$, and $L_{\tilde\psi_{t}}^{(k)}$ denote the Lipschitz constants of $\psi^{(k)}(\cdot,t)$, $\tilde \psi^{(k)}(\cdot,t)$, $\partial_{xx}\psi(\cdot,t)$, $\partial_{x}\psi(x,\cdot)$, $\partial_{xx}\tilde\psi(\cdot,t)$, and $\partial_{x}\tilde\psi(x,\cdot)$, respectively.
    Then, there exist constants $L_\psi$, $L_{\tilde\psi}$, $L_{\psi_{xx}}$, $L_{\psi_{t}}$, $L_{\tilde\psi_{xx}}$, and $L_{\tilde\psi_{t}}$ which satisfy
    \begin{align}
      L_q^{(k)} \le L_q < \infty,
    \end{align}
    where $q$ denotes $\psi$, $\tilde\psi$, $\psi_{xx}$, $\psi_{t}$, $\tilde\psi_{xx}$, $\tilde\psi_{t}$.
    That is, these constants can be assumed to be independent of $k$.
  }
\end{proposition}
\begin{proof}
  \rerev{
  The proof is immediately derived from the result of Lemma 2.4 of \cite{Cardaliaguet2017Learninga}, which states that the iterated variables $\left(m^{(k)}, v^{(k)}\right)$ are bounded in the H\"{o}lder space $\calC^{2+\eta, 1+\eta/2}$. 
  }
\end{proof}

\subsection{Convergence for the backward equation}

We start by studying the convergence property of the finite difference method applied to a backward equation.
\rev{Suppose that smooth $\bar m=\bar m(x,t)\ge 0$ and $v_T=v_T(x)$ are given, where the Lipschitz constant of $\bar m(\cdot,t)$ is denoted as $L_m$.}
Then, we consider the backward equation
\begin{align}
  \begin{split}\label{eq:ADR-fp-non-conservative_}
    \partial_t \psi(x ,t) &= -  f(x)\partial_x \psi(x,t) - \nu \partial_{xx}\psi(x,t) + \frac{1}{\lambda}\rev{g(x,\bar m(\cdot,t))}\psi(x,t),
  \end{split}
\end{align}
where $\psi(x,T) = \exp(-v_T(x)/\lambda)$.
For a given $\bar M= \qty(\bar M_{i,j})\ge 0$, we consider the finite difference scheme
\begin{align}
  \begin{split}    
    \frac{\Psi_{i,j-1}-\Psi_{i,j}}{\Delta t} &= f_i^{+} \frac{\Psi_{i+1,j}-\Psi_{i,j}}{\Delta x}+f_i^{-} \frac{\Psi_{i,j}-\Psi_{i-1,j}}{\Delta x}\\
    &\quad + \nu\frac{\Psi_{i+1,j}-2\Psi_{i,j}+\Psi_{i+1,j}}{\Delta x^2} - \frac{1}{\lambda}\rev{g\left(x_i, \bar M_{j}\right)}\Psi_{i,j-1},
  \end{split}\label{eq:upwind_ADR_}
\end{align}
where $\Psi_{i,N_t} = \exp(-v_T(x_i)/\lambda)$.
First, we give a bound for the solution of the scheme.
\begin{proposition}\label{prop:bound_Psi}
  Suppose that Assumption \ref{assmp:convergence_analysis} is satisfied.
  Then there exists $\Gamma_{\min}>0$ and $\Gamma_{\max}>0$ such that
  \begin{align}\label{eq:Psi_bound}
    \Gamma_{\min} \le \Psi_{i,j} \le \Gamma_{\max}.
  \end{align}
  Therein, $\Gamma_{\max}$ depends only on $\lambda$ and $v_T$, and $\Gamma_{\min}$ on $\lambda$, $v_T$, $T$, and $g$.
  In particular, $\Gamma_{\max}$ and $\Gamma_{\min}$ are independent of $\bar M_{i,j}$.
\end{proposition}
\begin{proof}
  By rearranging \eqref{eq:upwind_ADR_}, we obtain
  \begin{align}
    \begin{split}     
      \Psi_{i,j-1}&=\frac{1- \alpha\left|f_{i}\right| -2\beta \nu }{1+\frac{\Delta t}{\lambda} \rev{g(x_i, \bar M_{j})}} \Psi_{i,j}
      + \frac{\alpha f_i^{+} + \beta \nu}{1+\frac{\Delta t}{\lambda} \rev{g(x_i, \bar M_{j})}} \Psi_{i+1,j}
       + \frac{-\alpha f_i^{-} + \beta \nu}{1+\frac{\Delta t}{\lambda} \rev{g(x_i, \bar M_{j})}} \Psi_{i-1,j},
    \end{split}\label{eq:explicit_upwind_ADR_}
  \end{align}
  where we recall $\alpha=\Delta t/\Delta x$ and $\beta=\Delta t/\Delta x^2$.
  For the right-hand side of \eqref{eq:explicit_upwind_ADR_} to be positive, the following must hold for the first term:
  \begin{align}
    1- \alpha\left|f\right| -2\beta \nu  > 0.
  \end{align}
  This is guaranteed with the CFL condition \eqref{eq:CFL}.
  Next, we calculate $\Gamma_{\min}$ and $\Gamma_{\max}$.
  We first define 
  \begin{align}
    \Gamma_{v_T} &\coloneqq \inf_{x\in\Omega}\exp\qty(-\frac{v_T(x)}{\lambda}),\\
    \Gamma_{v_T}' &\coloneqq \sup_{x\in\Omega}\exp\qty(-\frac{v_T(x)}{\lambda}),
  \end{align}
  Then, for all $i$, we have
  \begin{align}
    \begin{split}      
      \Psi_{i,N_t-1}&\ge \qty{\frac{1- \alpha\left|f_{i}\right| -2\beta \nu }{1+\frac{\Delta t}{\lambda} \rev{g(x_i, \bar M_{j})}} + \frac{\alpha f_i^{+} + \beta \nu}{1+\frac{\Delta t}{\lambda} \rev{g(x_i, \bar M_{j})}} + \frac{-\alpha f_i^{-} + \beta \nu}{1+\frac{\Delta t}{\lambda} \rev{g(x_i, \bar M_{j})}}} \Gamma_{v_T}\\
      &=\frac{1}{1+\frac{\Delta t}{\lambda} \rev{g(x_i, \bar M_{j})}}\Gamma_{v_T}\\
      &\ge {\frac{1}{1+\frac{\Delta t}{\lambda} \|g\|_\infty}\Gamma_{v_T}},
    \end{split}
  \end{align}
  By sequentially applying this to $j=N_t,\ldots, 0$, the following inequality is established for all $i$ and $j$:
  \begin{align}\label{eq:def_varepsilon_Psi}
    \begin{split}      
      \Psi_{i,j}&\ge {\qty(\frac{1}{1+\frac{\Delta t}{\lambda} \|g\|_\infty})^{N_t}\Gamma_{v_T}
      \ge \exp\qty(\frac{T}{\lambda}\|g\|_\infty)^{-1}\Gamma_{v_T}
      \eqqcolon \Gamma_{\min}}.
    \end{split}
  \end{align}
  Similarly, for all $i$, we have
  \begin{align}
    \begin{split}      
      \Psi_{i,N_t-1}&\le \frac{1}{1+\frac{\Delta t}{\lambda} \rev{g(x_i, \bar M_{j})}}\Gamma_{v_T}'
      \le \Gamma_{v_T}'.
    \end{split}
  \end{align}
  By sequentially applying this to $j=N_t,\ldots, 0$, we have
  \begin{align}\label{eq:def_varepsilon_prime_Psi}
    \begin{split}      
      \Psi_{i,j}&\le \Gamma_{v_T}'\eqqcolon \Gamma_{\max},
    \end{split}
  \end{align}
  which ends the proof.
\end{proof}

\begin{proposition}\label{thm:convergence}
  \rerev{Suppose that Assumption \ref{assump:PDE-existence-uniform-fp-convergence}-(1)--(4) and Assumption \ref{assmp:convergence_analysis} are satisfied.}
  Set 
  \begin{align}\label{eq:assump_m_small}
    \varepsilon(\bar m, \bar M) \coloneqq \sup_{i,j} \qty|\bar m(x_i,t_j)-\bar M_{i,j}|.
  \end{align}
  Then, for the solution $\psi$ for \eqref{eq:ADR-fp-non-conservative_} and $\Psi$ for \eqref{eq:upwind_ADR_}, there exist constants $C_1,C_2>0$ that are independent of $\Delta x, \Delta t, \bar m, \bar M$, and we have
  \begin{align}\label{eq:psi_error_bound}
    \rev{\sup_{i,j} \qty|\psi(x_i,t_j)-\Psi_{i,j}| < C_1 \qty((\Delta x)^\eta+ (\Delta t)^{\frac{\eta}{2}})} + C_2\varepsilon(\bar m, \bar M).
  \end{align}
\end{proposition}
\begin{proof}
  Setting $e_{i,j} \coloneqq  \Psi_{i,j} - \psi(x_i, t_j)$, we have by 
  \eqref{eq:ADR-fp-non-conservative_} and \eqref{eq:upwind_ADR_} 
  \begin{align}
    \begin{split}      
      & \frac{e_{i,j-1} - e_{i,j}}{\Delta t} 
            -f_i^{+} \frac{e_{i+1,j}-e_{i,j}}{\Delta x}
      -f_i^{-} \frac{e_{i,j}-e_{i-1,j}}{\Delta x} - \nu \frac{e_{i+1,j}-2e_{i,j}+e_{i+1,j}}{\Delta x^2} +\frac{1}{\lambda} \rev{g\left(x_i, \bar M_{j}\right)}e_{i,j-1}\\
      &=\underbrace{\frac{1}{\lambda}\left[ \rev{g\left(x_i, \bar M_{j}\right)}(
      \psi(x_i,t_{j})-\psi(x_i,t_{j-1}))+
      (\rev{g\left(x_i, \bar m(\cdot,t_j)\right)} - \rev{g\left(x_i, \bar M_{j}\right)})\psi(x_i,t_j)\right]}_{=r^1_{i,j}} \\
      & 
      \mbox{ }\quad +\underbrace{\left(\frac{\psi(x_i,t_{j-1}) - \psi(x_i,t_j)}{\Delta t} +
      \partial_t\psi(x_i,t_j)\right)}_{=r^2_{i,j}}\\
      &\mbox{ }\quad +\underbrace{\left(
      f^+_i\frac{\psi(x_{i+1},t_{j}) - \psi(x_i,t_j)}{\Delta x} 
      +f^-_i\frac{\psi(x_{i},t_{j}) - \psi(x_{i-1},t_j)}{\Delta x} -
      f(x_i)\partial_x\psi(x_i,t_j)
      \right)}_{=r^3_{i,j}}\\
      &\mbox{ }\quad +\underbrace{\left(
      \nu\frac{\psi(x_{i+1},t_{j}) -2\psi(x_{i},t_{j})+ \psi(x_i,t_j)}{\Delta x^2} 
      -\nu \partial_{xx}\psi(x_i,t_j)
      \right)}_{=r^4_{i,j}}.
      \label{eq:eps_1}
    \end{split}
  \end{align}
  This is equivalently written as 
  \begin{multline}
   \qty[1 + \frac{\Delta t}{\lambda}  \rev{g\left(x_i, \bar M_{j}\right)}]e_{i,j-1}
  =(1-\alpha |f_i|-2\nu\beta)e_{i,j}\\
  +(\alpha f_i^++\nu\beta)e_{i+1,j}
  +(-\alpha f_i^-+\nu\beta)e_{i-1,j}+\Delta t
  (r^1_{i,j}+r^2_{i,j}+r^3_{i,j}+r^4_{i,j}).
  \end{multline}
  The standard error estimates for the difference quotients give
  \begin{equation}
  |r^2_{i,j}| \le \rev{(\Delta t)^{\frac{\eta}{2}}R_{1,t}},\quad 
  |r^3_{i,j}| \le (\Delta x)R_{1,x},\quad
  |r^4_{i,j}| \le \rev{(\Delta x)^\eta R_{2,x}}, 
  \end{equation}
  where 
  \begin{equation}
  R_{1,t}=\rev{L_{\psi_{t}}},\quad 
  R_{1,x}=\frac{1}{2}\|f\partial_{xx}\psi\|_{\infty},\quad 
  R_{2,x}=\rev{L_{\psi_{xx}}}.
  \end{equation}
  Moreover,  
  \begin{align}
    \begin{split}\label{eq:r1}   
      |r^1_{i,j}| 
      &\le (\Delta t) R_{2,t} + \rev{\bfd_1(\bar M_{j},\bar m(\cdot,t_j)) \frac{\|\psi\|_\infty}{\lambda}
      \|\partial_m g\|_\infty L_g}\\
      &\le (\Delta t) R_{2,t} + \rev{\|\bar M_{j}-\bar m(\cdot,t_j)\|_\infty R_m}\\
      &\le (\Delta t) R_{2,t} + \varepsilon(\bar m, \bar M) R_m + \rev{(\Delta x) L_m},
    \end{split}
  \end{align}
  where 
  \begin{align}
  \rev{
    R_{2,t}=\frac{\|\partial_t\psi\|_\infty }{\lambda}\|g\|_\infty, \
    R_m=\frac{\|\psi\|_\infty}{\lambda}
    \|\partial_m g\|_\infty L_g C_\Omega.
  }
  \end{align}
  \rev{In the first line of \eqref{eq:r1}, we have used the Lipschitz continuity of $g$ in Assumption \ref{assump:PDE-existence-uniform-fp-convergence}, where $L_g$ denotes Lipschitz constant of $g$.
  Then, in the second line, we have applied Proposition \ref{prop:d1-sup}.
  Finally, we used the smoothness assumption of $\bar m$, where $L_m$ denotes Lipschitz constant of $\bar m$.}
  Summing up, we deduce
  \begin{align}
    |r^1_{i,j}+r^2_{i,j}+r^3_{i,j}+r^4_{i,j}|
    \le \rev{R:=(\Delta t)^{\frac{\eta}{2}} R_{1,t} + (\Delta t)R_{2,t} 
    +(\Delta x)(R_{1,x}+L_m) + (\Delta x)^{\eta}R_{2,x}}
    + \varepsilon(\bar m, \bar M) R_m.
  \end{align}
  Consequently, setting $E_j=\max_{i}|e_{i,j}|$, we have by \eqref{eq:eps_1} and \eqref{eq:CFL}
  \begin{align}
    \begin{split}      
        |e_{i,j-1}|
        &\le \left(1+\frac{\Delta t}{\lambda} \rev{g\left(x_i, \bar M_{j}\right)}\right)^{-1} 
        \left[(1-\alpha |f_i|-2\nu\beta)|e_{i,j}|
        \right. \\
        & { }\quad \left. +(\alpha f_i^++\nu\beta)|e_{i+1,j}| +(-\alpha f_i^-+\nu\beta)|e_{i-1,j}|+\Delta t R
        \right]\\
        &\le E_{j}+\Delta t R,
    \end{split}
  \end{align}
  since $\left(1+\frac{\Delta t}{\lambda} \rev{g\left(x_i, \bar M_{j}\right)}\right)^{-1}<1$. 
  This implies
  \begin{equation}
  E_{j-1} \le E_{j}+\Delta t R\le \cdots \le E_{N_t}+(N_t-j+1)\Delta t R\le TR,
  \end{equation}
  for any $j=1,\ldots,N_t$. 
  Since $E_{N_t} = 0$, setting 
  \begin{align}
    C_1 &\coloneqq \rev{T\max\{  R_{1,x}+L_m,R_{2,x},R_{1,t},R_{2,t}\}},\\
    C_2 &\coloneqq T R_m,
  \end{align}
  completes the proof of the desired estimate \eqref{eq:psi_error_bound}.  

\end{proof}

\subsection{Convergence for the forward equation}

We proceed to an examination of the convergence property of the finite difference method applied to a forward equation. 
For a given smooth function $\bar{m}=\bar{m}(x,t)$ and $m_0 =m_0(x)$, we consider the forward equation
\begin{align}
  \begin{split}\label{eq:ADR-fp-conservative_}
    \partial_t \tilde \psi(x ,t) &= - \partial_x( f(x) \tilde \psi(x,t)) + \nu \partial_{xx}\tilde \psi(x,t) - \frac{1}{\lambda}\rev{g(x,\bar m(\cdot,t))} \tilde \psi(x,t),
  \end{split}
\end{align}
where $\tilde{\psi}(x,0)=m_0(x)/\psi(x,0)$ with the solution $\psi$ of \eqref{eq:ADR-fp-non-conservative_}. 
For a given $M=(M_{i,j})$, we consider the finite difference scheme
\begin{align}
  \begin{split}    
    \frac{\tilde\Psi_{i,j+1}-\tilde\Psi_{i,j}}{\Delta t} &= -\frac{1+\sgn f_i}{2} \frac{f_{i}\tilde\Psi_{i,j} - f_{i-1}\tilde\Psi_{i-1,j}}{\Delta x}
     - \frac{1-\sgn f_i}{2}  \frac{f_{i+1}\tilde\Psi_{i+1,j} - f_{i}\tilde\Psi_{i,j}}{\Delta x}\\
    &\quad+ \nu\frac{\tilde\Psi_{i+1,j}-2\tilde\Psi_{i,j}+\tilde\Psi_{i+1,j}}{\Delta x^2} - \frac{1}{\lambda}\rev{g\left(x_i, \bar M_{j}\right)}\tilde\Psi_{i,j+1},
  \end{split}\label{eq:upwind_ADR_conservative_}
\end{align}
where $\tilde{\Psi}_{i,0}=m_0(x_i)/\psi(x_i,0)$.

\begin{proposition}\label{prop:bound_tildePsi}
  Suppose that Assumption \ref{assmp:convergence_analysis} is satisfied.
  Then there exists $\tilde{\Gamma}_{\max}>0$ such that 
  \begin{align}
    0< \tilde\Psi_{i,j}&\le \tilde \Gamma_{\max},
  \end{align}
  where $\tilde \Gamma_{\max}>0$ depends only on $\lambda$, $v_T$, $T$, $g$, and $m_0$;
  it is independent of $\bar M_{i,j}$. 
\end{proposition}
\begin{proof}
  By rearranging \eqref{eq:upwind_ADR_conservative_}, we obtain
  \begin{align}
    \begin{split}
      \tilde\Psi_{i,j+1}&=\frac{1- \alpha\left|f_{i}\right| -2\beta \nu }{1+\frac{\Delta t}{\lambda} \rev{g(x_i, \bar M_{j})}} \tilde\Psi_{i,j}
      + \frac{-\alpha \frac{1-\sgn f_i}{2}f_{i+1} + \beta \nu}{1+\frac{\Delta t}{\lambda} \rev{g(x_i, \bar M_{j})}} \tilde\Psi_{i+1}^{j}
       + \frac{\alpha \frac{1+\sgn f_i}{2}f_{i-1}  + \beta \nu}{1+\frac{\Delta t}{\lambda} \rev{g(x_i, \bar M_{j})}} \tilde\Psi_{i-1,j}.
    \end{split}\label{eq:explicit_upwind_ADR_conservative}
  \end{align}
  Thanks to the CFL condition \eqref{eq:CFL}, the right-hand side is non-negative. 
  Introducing 
  \begin{align}
    \tilde{\Gamma}_{m_0} \coloneqq 
     \frac{\sup_x m_0(x)}{\Gamma_{\min}},
  \end{align}
  with $\Gamma_{\min}$ defined in \eqref{eq:def_varepsilon_Psi}, we can perform an estimation as 
  \begin{align}
    \begin{split}      
      \tilde\Psi_{i,1}
      &\le \qty{\qty(1- \alpha\left|f_{i}\right|) 
      + \qty(-\alpha \frac{1-\sgn f_i}{2}f_{i+1}) 
      + \qty(\alpha \frac{1+\sgn f_i}{2}f_{i-1})}\tilde{\Gamma}_{m_0}\\
      &\le \qty(1+ \alpha\max_{j\in\{i\pm 1\}}\left|f_{i} - f_{j}\right|)\tilde{\Gamma}_{m_0}\\
      &\le \qty(1+ \Delta t L_f)\tilde{\Gamma}_{m_0},
    \end{split}
  \end{align}
  where $L_f$ denotes Lipschitz constant of $f$ in $\Omega$. 
  Consequently, 
  \begin{equation}
    \tilde\Psi_{i,j}\le \qty(1+ \Delta t L_f)^{N_t}\tilde{\Gamma}_{m_0}
    \le \exp(TL_f)\tilde{\Gamma}_{m_0}
    \eqqcolon \tilde{\Gamma}_{\max},
  \end{equation}\label{eq:tilGamma_max}
  which completes the proof.
\end{proof}

\begin{proposition}\label{thm:convergence-conservative}
  \rerev{Suppose that Assumption \ref{assump:PDE-existence-uniform-fp-convergence}-(1)--(4), \ref{assump:LHJB}, \ref{assmp:convergence_analysis} are satisfied.}
  In addition, suppose that \eqref{eq:assump_m_small} holds.
  Then, for the solution $\tilde\psi$ for \eqref{eq:ADR-fp-conservative_} and $\tilde\Psi$ for \eqref{eq:upwind_ADR_conservative_}, 
  there exist constants $\tilC_1, \tilC_2>0$ that are independent of $\Delta x, \Delta t, \bar m, \bar M$ such that
  \begin{align}\label{eq:tilde_psi_error_bound}
    \rev{\sup_{i,j} \qty|\tilde\psi(x_i,t_j)-\tilde\Psi_{i,j}| < \tilC_1 \qty((\Delta x)^\eta+ (\Delta t)^{\frac{\eta}{2}})} + \tilC_2\varepsilon(\bar m, \bar M).
  \end{align}
\end{proposition}
\begin{proof}
  Set $\tilde e_{i,j} \coloneqq  \tilde\Psi_{i,j} - \tilde\psi(x_i, t_j)$.
  First, we estimate the error $\tilde e_{i,0}$ at the initial time.
  By using \eqref{eq:assump_m_small} and Proposition \ref{thm:convergence}, we obtain 
  \begin{align}
    \begin{split}\label{eq:tile0_bound}
      \qty|\tilde\Psi_{i,0} - \tilde \psi(x_i,0) | 
      &=  m_0(x_i)\qty| \frac{1}{\Psi_{i,0}} -\frac{1}{\psi(x_i,0)}|\\
      &\le \frac{m_0(x_i)}{\gamma_{\min}\Gamma_{\min}}|e_{i,0}|
      \le \frac{m_0(x_i)}{\gamma_{\min}\Gamma_{\min}}\qty( C_1 \rev{\qty((\Delta x)^\eta + (\Delta t)^\frac{\eta}{2})} + C_2\varepsilon(\bar m,\bar M )),
    \end{split}
  \end{align}
  where we have used \eqref{eq:psi_bound} and \eqref{eq:Psi_bound}.
  Letting $\tilde \psi_{i,j}=\tilde \psi(x_i,t_j)$, we observe 
  \begin{align}
    \begin{split}      
      &\frac{\tilde e_{i,j+1}-\tilde e_{i,j}}{\Delta t} 
      +\frac{1+\sgn f_i}{2} \frac{f_{i}\tilde e_{i,j} - f_{i-1}\tilde e_{i-1,j}}{\Delta x}
          + \frac{1-\sgn f_i}{2}  \frac{f_{i+1}\tilde e_{i+1,j} - f_{i}\tilde e_{i,j}}{\Delta x}\\
      &   { }\qquad  -\nu\frac{\tilde e_{i+1,j}-2\tilde e_{i,j}+\tilde e_{i+1,j}}{\Delta x^2} + \frac{1}{\lambda}\rev{g\left(x_i, \bar M_{j}\right)}\tilde e_{i,j+1}\\
      &=  
      \underbrace{
      \frac{1}{\lambda}\left[ \rev{g\left(x_i, \bar M_{j}\right)}(\tilde \psi_{i,j}
      -\tilde \psi_{i,j+1})
      +(\rev{g\left(x_i, \bar m(\cdot,t_j)\right)}
      - \rev{g\left(x_i, \bar M_{j}\right)})\tilde \psi_{i,j}\right]}_{=\tilr^1_{i,j}}\\
      & \underbrace{-\frac{\tilde \psi_{i,j+1}-\tilde \psi_{i,j}}{\Delta t} 
      +\partial_t\tilde\psi (x_i,t_j)}_{=\tilr^2_{i,j}}\\
      &\underbrace{- \frac{1+\sgn f_i}{2} \frac{f_{i}\tilde \psi_{i,j} - f_ {i-1}\tilde \psi_{i-1,j}}{\Delta x}
          - \frac{1-\sgn f_i}{2}  \frac{f_{i+1}\tilde \psi_{i+1,j} - f_{i}\tilde \psi_{i,j}}{\Delta x}+(\partial_x(f\tilde\psi)) (x_i,t_j)}_{=\tilr^3_{i,j}}\\
      &\underbrace{  +\nu\frac{\tilde \psi_{i+1,j}-2\tilde \psi_{i,j}+\tilde \psi_{i+1,j}}{\Delta x^2} +-\nu\partial_{xx}\tilde\psi (x_i,t_j)}_{=\tilr^4_{i,j}}.
    \end{split}\label{eq:tileps_1}
  \end{align}
  This is equivalently written as 
  \begin{multline}
   \qty[1 + \frac{\Delta t}{\lambda}  \rev{g\left(x_i, \bar M_{j}\right)}]\tile_{i,j+1}
  =(1-\alpha |f_i|-2\nu\beta)\tile_{i,j}\\
  % +(\alpha f_i^++\nu\beta)\tile_{i+1,j}
  +\qty(\alpha \frac{1-\sgn f_i}{2}f_{i+1} +\nu\beta)\tile_{i+1,j}
  +\qty(-\alpha \frac{1+\sgn f_i}{2} f_{i-1}+\nu\beta)\tile_{i-1,j}\\
  +\Delta t
  (\tilr^1_{i,j}+\tilr^2_{i,j}+\tilr^3_{i,j}+\tilr^4_{i,j}).
  \end{multline}
  The standard error estimates for the difference quotients give
  \begin{equation}
  |\tilr^2_{i,j}| \le \rev{(\Delta t)^{\frac{\eta}{2}}\tilR_{1,t}},\quad 
  |\tilr^3_{i,j}| \le (\Delta x)\tilR_{1,x},\quad
  |\tilr^4_{i,j}| \le \rev{(\Delta x)^\eta \tilR_{2,x}}, 
  \end{equation}
  where 
  \begin{equation}
  \tilR_{1,t}=\rev{L_{\tilde\psi_{t}}},\quad 
  \tilR_{1,x}=\frac{1}{2}\|\partial_{xx}(f\tilde\psi)\|_{\infty},\quad 
  \tilR_{2,x}=\rev{L_{\tilde\psi_{xx}}}.
  \end{equation}
  Moreover,  
  \begin{align}
    \begin{split}      
      |\tilr^1_{i,j}| 
      &\le (\Delta t) \tilR_{2,t} + \rev{\bfd_1 (\bar M_{j},\bar m(\cdot,t_j)) \frac{\|\tilde\psi\|_\infty}{\lambda}\|\partial_m g\|_\infty L_g}\\
      &\le (\Delta t) \tilR_{2,t} + \rev{\|\bar M_{j}-\bar m(\cdot,t_j)\|_\infty \tilR_m}\\
      &\le (\Delta t) \tilR_{2,t} + \varepsilon(\bar m, \bar M) \tilR_m + \rev{(\Delta x)L_m},
    \end{split}
  \end{align}
  where 
  \begin{align}
  \rev{
    \tilR_{2,t} = \frac{\|\partial_t\tilde\psi\|_\infty}{\lambda}\|g\|_\infty, \
    \tilR_m=\frac{\|\tilde\psi\|_\infty}{\lambda}
    \|\partial_m g\|_\infty L_g C_\Omega.
  }
  \end{align}
  Summing up, we deduce
  \begin{align}
    |r^1_{i,j}+r^2_{i,j}+r^3_{i,j}+r^4_{i,j}|
    \le \rev{\tilR:=(\Delta t)^{\frac{\eta}{2}} \tilR_{1,t} + (\Delta t)\tilR_{2,t} 
    +(\Delta x)(\tilR_{1,x} + L_m) + (\Delta x)^{\eta}\tilR_{2,x}}
    + \varepsilon(\bar m, \bar M) \tilR_m.
    \end{align}
  Consequently, setting $\tilE_j=\max_{i}|\tile_{i,j}|$, we have by \eqref{eq:tileps_1} and \eqref{eq:CFL}
  \begin{align}
    \begin{split}
      |\tile_{i,j+1}|&\le \left(1+ \alpha\sup_{k\in{i\pm1}} \left|f_{i} - f_{k}\right| \right) \tilE_j + \Delta t\tilR,\\
      &\le \left(1+ \Delta t L_f \right) \tilE_j + \Delta t\tilR,
    \end{split}
  \end{align}
  since $\left(1+\frac{\Delta t}{\lambda} \rev{g\left(x_i, \bar M_{j}\right)}\right)^{-1}<1$. 
  Here, $L_f$ denotes the Lipschitz constant of $f$ in $\Omega$.
  We then obtain by \eqref{eq:tile0_bound}
  \begin{align}
    \begin{split}
      \tilE_{j+1} &\le \left(1+ \Delta t L_f \right)\tilE_{j}+\Delta t \tilR\\
      &\le \cdots \le \left(1+ \Delta t L_f \right)^{N_t}\tilE_{0}+\Delta t \tilR\left(1+ (1 + \Delta t L_f ) + \cdots (1 + \Delta t L_f )^{N_t}\right)\\
      &\le \exp(TL_f)\frac{\|m_0\|_\infty}{\gamma_{\min}\Gamma_{\min}}\qty( C_1 \rev{\qty((\Delta x)^\eta + (\Delta t)^\frac{\eta}{2})} + C_2\varepsilon(\bar m,\bar M ))
      + \frac{\exp(TL_f)-1}{L_f}\tilR,
    \end{split}
  \end{align}
  for any $j=1,\ldots,N_t$. 
  Finally, setting 
  \begin{align}
  \tilC_1 &\coloneqq \exp(TL_f)\frac{C_1\|m_0\|_\infty}{\gamma_{\min}\Gamma_{\min}} + 
  \frac{\exp(TL_f)-1}{L_f} \rev{\max\{ \tilR_{1,x} + L_m,\tilR_{2,x},\tilR_{1,t},\tilR_{2,t}\}},\\
  \tilC_2 &\coloneqq \exp(TL_f)\frac{C_2\|m_0\|_\infty}{\gamma_{\min}\Gamma_{\min}} + 
  \frac{\exp(TL_f)-1}{L_f} \tilR_m,
  \end{align}
  we complete the proof of the desired estimate \eqref{eq:tilde_psi_error_bound}.  
\end{proof}

\subsection{Convergence for the \rev{\rerev{fictitious} play iteration}}

\begin{proof}[Proof of Theorem \ref{thm:convergence-repeat}]
  We apply Propositions \ref{thm:convergence} and \ref{thm:convergence-conservative}.
  We define
  \begin{align}
    \delta^{(k)}&\coloneqq \sup_{i,j} |m^{(k)}(x_i,t_j) - M_{i,j}^{(k)}|,\\
    \begin{split}
      \varepsilon^{(k)}&\coloneqq \sup_{i,j} |\bar m^{(k)}(x_i,t_j)-\bar M_{i,j}^{(k)}|
      =\sup_{i,j} \left|\frac{1}{k+1}\sum_{\ell=0}^k m^{(\ell)}(x_i,t_j)- \frac{1}{k+1}\sum_{\ell=0}^k M_{i,j}^{(\ell)}\right|,
    \end{split}
  \end{align}
  then we have
  \begin{align}\label{eq:delta_epsilon}
    \varepsilon^{(k)}\le \frac{1}{k+1} \sum_{\ell=0}^k \delta^{(\ell)}.
  \end{align}
  By using the results in Propositions \ref{thm:convergence} and \ref{thm:convergence-conservative}, we derive
  \begin{align}
    \begin{split}\label{eq:delta_bound}
      \delta^{(k)}&= \sup_{i,j}\qty|\psi^{(k)}(x_i,t_j)\tilde\psi^{(k)}(x_i,t_j)-\Psi_{i,j}^{(k)}\tilde\Psi_{i,j}^{(k)}|\\
      &\le \sup_{i,j}\qty|\psi^{(k)}(x_i,t_j)\qty{\tilde\psi^{(k)}(x_i,t_j)-\tilde\Psi_{i,j}^{(k)}}  + \tilde\Psi_{i,j}^{(k)}\qty{\psi^{(k)}(x_i,t_j) - \Psi_{i,j}^{(k)}} |\\
      &\le \sup_{i,j}\left\{\qty|\psi^{(k)}(x_i,t_j)|\qty(\tilC_1 \rev{\qty( (\Delta x)^\eta + (\Delta t)^{\frac{\eta}{2}})} + \tilC_2\varepsilon^{(k-1)}(\bar m,\bar M ))\right.\\
      &\quad\quad + \left.\qty|\tilde\Psi_{i,j}^{(k)}| \qty(C_1 \rev{\qty( (\Delta x)^\eta + (\Delta t)^{\frac{\eta}{2}})} + C_{2}\varepsilon^{(k-1)}(\bar m,\bar M ))\right\}\\
      &\le \bar C_2\varepsilon^{(k-1)}(\bar m,\bar M ) + \bar C_1\rev{\qty( (\Delta x)^\eta + (\Delta t)^{\frac{\eta}{2}})}\\
      &\le \frac{\bar C_2}{k} \sum_{\ell=0}^{k-1} \delta^{(\ell)} + \bar C_1\rev{\qty( (\Delta x)^\eta + (\Delta t)^{\frac{\eta}{2}})},
    \end{split}
  \end{align}
  where we have defined
  \begin{align}
    \bar C_1 &\coloneqq\gamma_{\max}\tilC_1 + \tilde\Gamma_{\max} C_1,\\
    \bar C_2 &\coloneqq \gamma_{\max}\tilC_2 + \tilde\Gamma_{\max} C_2.
  \end{align}

  Now we define $\bar \delta^{(0)}\coloneqq \delta^{(0)}=0$ and 
  \begin{align}
    \bar \delta^{(k)} \coloneqq \bar C_2\left(\max _{\ell \in\{0, \ldots, k-1\}} \delta^{(\ell)}\right) + \bar C_1\rev{\qty((\Delta x)^\eta + (\Delta t)^\frac{\eta}{2})},
  \end{align}
  then, from the definition we have
  \begin{align}
    \begin{split}
      \bar \delta^{(k)} &\le \bar C_2\left(\max _{\ell \in\{0, \ldots, k\}} \delta^{(\ell)}\right) + \bar C_1\rev{\qty((\Delta x)^\eta + (\Delta t)^\frac{\eta}{2})}
      =\bar \delta^{(k+1)}.
    \end{split}
  \end{align}
  Thus, 
  \begin{align}
    \begin{split}
      \delta^{(k)} & \leq \bar{\delta}^{(k)} \\
      &=\bar{C}_{2}\left(\max _{\ell \in\{0, \ldots, k-1\}} \delta^{(\ell)}\right)+\bar{C}_1 \rev{\qty((\Delta x)^\eta + (\Delta t)^\frac{\eta}{2})} \\
      &\le\bar{C}_{2}\left(\max _{\ell \in\{0, \ldots, k-1\}} \bar\delta^{(\ell)}\right)+\bar{C}_1 \rev{\qty((\Delta x)^\eta + (\Delta t)^\frac{\eta}{2})} \\
      & \leq \bar{C}_{2} \bar{\delta}^{(k-1)}+\bar{C}_1 \rev{\qty((\Delta x)^\eta + (\Delta t)^\frac{\eta}{2})} \\
      & \leq \bar{C}_{2}\left(\bar{C}_{2} \bar{\delta}^{(k-2)}+\bar{C}_1 \rev{\qty((\Delta x)^\eta + (\Delta t)^\frac{\eta}{2})}\right)+\bar{C}_1 \rev{\qty((\Delta x)^\eta + (\Delta t)^\frac{\eta}{2})} \\
      & \leq \bar{C}_{2}^{k} \bar{\delta}^{(0)}+\bar{C}_1 \rev{\qty((\Delta x)^\eta + (\Delta t)^\frac{\eta}{2})} \sum_{\ell=0}^{k-1} \bar{C}_{2}^{\ell} \\
      & = \bar{C}_{2}^{k} \bar{\delta}^{(0)}+ \frac{\bar{C}_{2}^{k}-1}{\bar{C}_{2}-1}\bar{C}_1 \rev{\qty((\Delta x)^\eta + (\Delta t)^\frac{\eta}{2})}\\
      & = \frac{\bar{C}_{2}^{k}-1}{\bar{C}_{2}-1}\bar{C}_1 \rev{\qty((\Delta x)^\eta + (\Delta t)^\frac{\eta}{2})}.
    \end{split}
  \end{align}
  Consequently, by letting
  \begin{align}
    \rev{\bar K(k)}&\coloneqq \bar{C}_{1}\frac{\bar{C}_{2}^{k}-1}{\bar{C}_{2}-1},
  \end{align}
  we have 
  \begin{align}
    \rev{\delta^{(k)}} &\le \rev{\bar K(k)}\rev{\qty((\Delta x)^\eta + (\Delta t)^\frac{\eta}{2})},\\
    \begin{split}
      \varepsilon^{(k)} &\le \frac{1}{k+1} \sum_{\ell=0}^k \delta^{(\ell)}\\
      &\le \frac{1}{k+1} \sum_{\ell=0}^k \qty{ \rev{\bar K(\ell)}\rev{\qty((\Delta x)^\eta + (\Delta t)^\frac{\eta}{2})}},\\
      &\le \rev{\bar K(k)}\rev{\qty((\Delta x)^\eta + (\Delta t)^\frac{\eta}{2})}.
    \end{split}
  \end{align}
  Using Proposition \ref{thm:convergence}, we obtain \eqref{eq:psi_convergence} as
  \begin{align}
    \begin{split}\label{eq:psi-repeat-bound}
      \sup_{i,j} \qty|\psi^{(k+1)}(x_i,t_j)-\Psi^{(k+1)}_{i,j}|
      &< C_1 \rev{\qty((\Delta x)^\eta + (\Delta t)^\frac{\eta}{2})} + C_2\varepsilon^{(k)}\\
      &\le \qty{C_1 + C_2\rev{\bar K(k)}}\rev{\qty((\Delta x)^\eta + (\Delta t)^\frac{\eta}{2})}\\
      &\eqqcolon \rev{K(k)}\rev{\qty((\Delta x)^\eta + (\Delta t)^\frac{\eta}{2})}.
    \end{split}
  \end{align}
  Similarly, we obtain \eqref{eq:tilpsi_convergence} by using Proposition \ref{thm:convergence-conservative}:
  \begin{align}
    \begin{split}
      \sup_{i,j} \qty|\tilde\psi^{(k+1)}(x_i,t_j)-\tilde\Psi^{(k+1)}_{i,j}|
      &< \tilC_1 \rev{\qty((\Delta x)^\eta + (\Delta t)^\frac{\eta}{2})} + \tilC_2\varepsilon^{(k)}\\
      &\le \qty{\tilC_1 + \tilC_2\rev{\bar K(k)}}\rev{\qty((\Delta x)^\eta + (\Delta t)^\frac{\eta}{2})}\\
      &\eqqcolon \rev{\tilde K(k)}\rev{\qty((\Delta x)^\eta + (\Delta t)^\frac{\eta}{2})}.
    \end{split}
  \end{align}
  This ends the proof.
\end{proof}

\begin{proof}[Proof of Theorem \ref{thm:convergence-repeat-limit}]
  Let $(m,v)$ be the solution to \eqref{eq:LQHJB}--\eqref{eq:LQFP_initial},
  $(m^{(k)},v^{(k)})$ be the solution to \eqref{eq:fp_LQHJB}--\eqref{eq:fp_LQFP_initial}, $(\psi^{(k)},\tilde\psi^{(k)})$ be the solution to \eqref{eq:ADR-fp-non-conservative}--\eqref{eq:ADR-fp-conservative_initial}, 
  and $(\Psi^{(k)},\tilde\Psi^{(k)})$ be the solution to \eqref{eq:upwind_ADR}--\eqref{eq:upwind_ADR_conservative_initial}.
  Let $\epsilon>0$.
  We first observe
  \begin{align}
    \begin{split}      
      \left|m(x_i,t_j)-\Psi_{i,j}^{(k)}\tilde\Psi_{i,j}^{(k)}\right|
      &\le \qty|m(x_i,t_j) - m^{(k)}(x_i,t_j) | +\qty| m^{(k)}(x_i,t_j) - \Psi_{i,j}^{(k)}\tilde\Psi_{i,j}^{(k)}|\\
      &\le \qty|m(x_i,t_j) - m^{(k)}(x_i,t_j) | + \rev{\bar K(k)}\rev{\qty((\Delta x)^\eta + (\Delta t)^\frac{\eta}{2})}.
    \end{split}
  \end{align}  
  Using Proposition \ref{prop:PDE-existence-uniform-fp-convergence}, there exists $k'\in \bbN$ such that 
  \begin{align}\label{eq:m_convergence1}
    \sup_{i,j} \qty|m(x_i,t_j) - m^{(k)}(x_i,t_j) |< \frac{\epsilon}{2},
  \end{align}
  for $k\ge k'$.
  Therefore, for $k\ge k'$, 
  \begin{align}
    \left|m(x_i,t_j)-\Psi_{i,j}^{(k)}\tilde\Psi_{i,j}^{(k)}\right| \le \frac{\epsilon}{2} + \rev{\bar K(k)}\rev{\qty((\Delta x)^\eta + (\Delta t)^\frac{\eta}{2})},
  \end{align}
  which implies \eqref{eq:m_convergence_repeat}.
  To prove \eqref{eq:v_convergence_repeat}, we observe 
  \begin{align}
    \begin{split}
      \left|v(x_i,t_j)-\qty(-\lambda \ln \Psi_{i,j}^{(k)}) \right| 
      &\le \qty|v(x_i,t_j)- v^{(k)}(x_i,t_j)| 
      + \qty| \qty(-\lambda \ln \psi^{(k)}(x_i,t_j)) -\qty(-\lambda \ln \Psi_{i,j}^{(k)}) |.
    \end{split}
  \end{align}
  By using Proposition \ref{prop:PDE-existence-uniform-fp-convergence}, we have 
  \begin{align}\label{eq:v_convergence1}
    \sup_{i,j} \qty|v(x_i,t_j) - v^{(k)}(x_i,t_j) |< \frac{\epsilon}{2},
  \end{align}
  for $k\ge k''$ with some $k''\in\bbN$.
  On the other hand, we have the following bound:
  \begin{align}
    \begin{split}
      \qty|\lambda \ln \psi^{(k)}(x_i,t_j) - \lambda \ln \Psi_{i,j}^{(k)}| 
      &= \lambda \qty|\ln\frac{\psi^{(k)}(x_i,t_j)}{\Psi_{i,j}^{(k)}}|\\
      &= \lambda \qty|\ln\qty{ \frac{\psi^{(k)}(x_i,t_j) - \Psi_{i,j}^{(k)}}{\Psi_{i,j}^{(k)}} + 1}|\\
      &\le \lambda \ln\qty{ \qty|\frac{\psi^{(k)}(x_i,t_j) - \Psi_{i,j}^{(k)}}{\Psi_{i,j}^{(k)}}| + 1}\\
      &\le \lambda \ln\qty{ \frac{\qty|\psi^{(k)}(x_i,t_j) - \Psi_{i,j}^{(k)}|}{\Gamma_{\min}} + 1}.
    \end{split}
  \end{align}
  Therefore, we deduce
  \begin{align}
    \left|v(x_i,t_j)-\qty(-\lambda \ln \Psi_{i,j}^{(k)}) \right| \le \frac{\epsilon}{2} + \lambda \ln\qty{ \frac{\qty|\psi^{(k)}(x_i,t_j) - \Psi_{i,j}^{(k)}|}{\Gamma_{\min}} + 1},
  \end{align}
  for $k\ge k''$.
  This completes the proof of \eqref{eq:v_convergence_repeat}.
\end{proof}

\section{Numerical experiments}\label{sec:numeric}

In this section, we conduct numerical experiments using the scheme proposed in Section \ref{sec:scheme}.
Although the convergence analysis we have performed is only for the one-dimensional case, we can numerically verify the results in both one- and two-dimensional cases.

\begin{example}  
  First, consider the one-dimensional interval $\Omega=[-1,1]$ with a periodic boundary condition, and consider the temporal parameter $T=1$.
  The functions in the LS-MFG equations \eqref{eq:LQHJB}--\eqref{eq:LQFP_initial} are set as follows:
  \begin{align}
    \begin{split}
      f(x) &= A x,\  g(x,m) = Q x^2 + \rev{L_1  \min\{m(x), L_2\}},\\
      \sigma &= 0.1,\  B = 1.0, \ R= 1.0,\\
      v_T(x) &= 0.0,\  m_0(x) = \frac{1}{\sqrt{2\pi \nu_0}}\exp\left(-\frac{x^2}{2 \nu_0}\right),
    \end{split}
  \end{align}
  where $A$, $Q$, \rev{$L_1$, $L_2$} and $\nu_0$ are scalar parameters.
  We consider the case with $A=1$, $Q=5$, \rev{$L_1\in\{0,0.1\}$, $L_2=10$}, and $\nu_0 =0.2$.
  These parameter settings imply the following scenarios.
  First, since $A>0$, the density distribution diverges around $x=0.0$ when the system has zero input.
  To prevent this, we would like to stabilize the system at the origin by setting $Q>0$.
  On the other hand, we also try to achieve congestion mitigation by setting $L_1>0$.
  \rev{Note that the function $g$ does not satisfy the regularity assumption made in Assumption \ref{assump:PDE-existence-uniform-fp-convergence}.
  Although defining the function as a convolution with a mollifier would guarantee regularity, we adopt this local coupling for convenience in conducting numerical experiments and for practicality in the control problem.
  Since the previous study has shown the convergence of fixed point iteration in MFG with local coupling~\cite{Gueant2012Mean}, we expect that the convergence of fictitious play may hold under weaker assumptions.}
  We define the discrete parameters as $\Delta x= 0.01$ and $\Delta t = 0.0025$.
  Then the CFL condition \eqref{eq:CFL} is satisfied:
  \begin{align}
    \alpha \|f\|_\infty + 2 \beta \nu = 0.5 < 1.
  \end{align}
  
  \begin{figure}[ht]
    \begin{subfigure}{0.49\textwidth}
      \includegraphics[width=\linewidth]{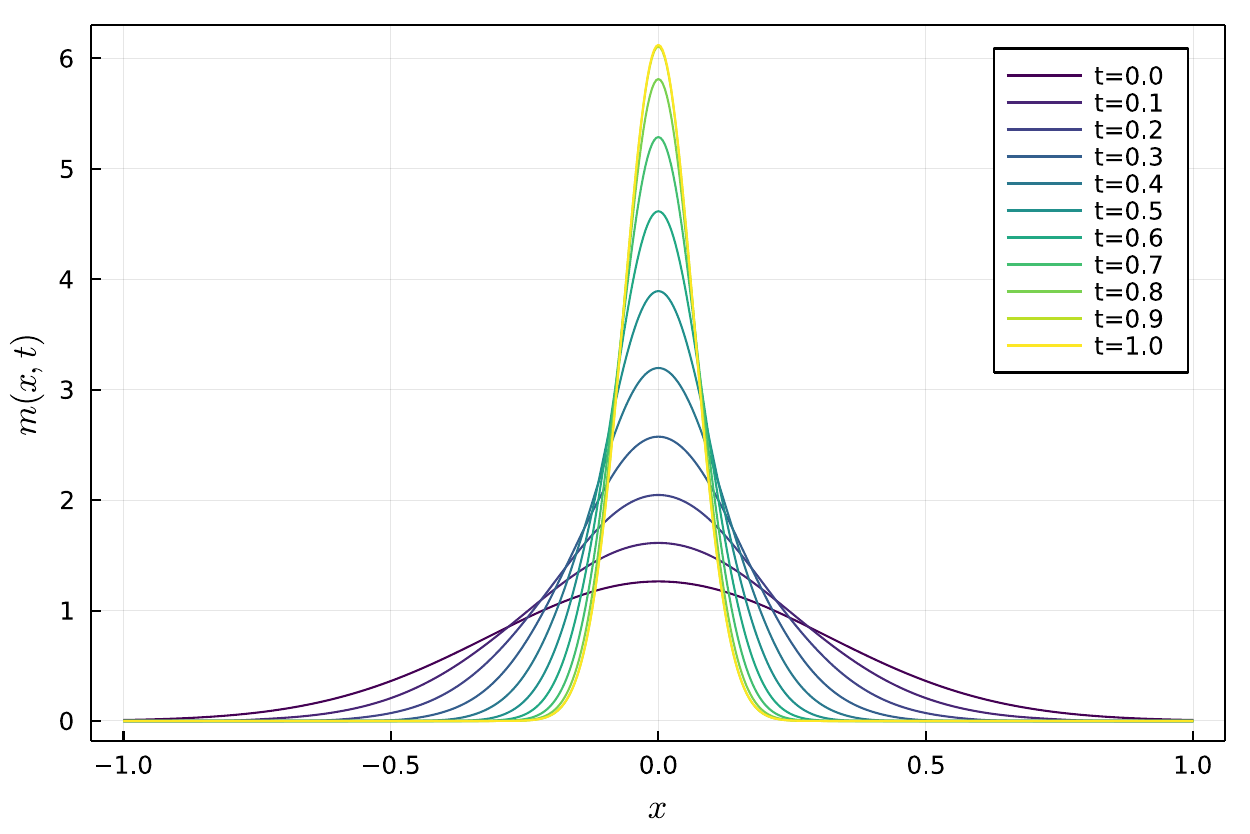}
      \caption{Density $m$.}
      \label{fig:test2-1m}
    \end{subfigure}
    \begin{subfigure}{0.49\textwidth}
      \includegraphics[width=\linewidth]{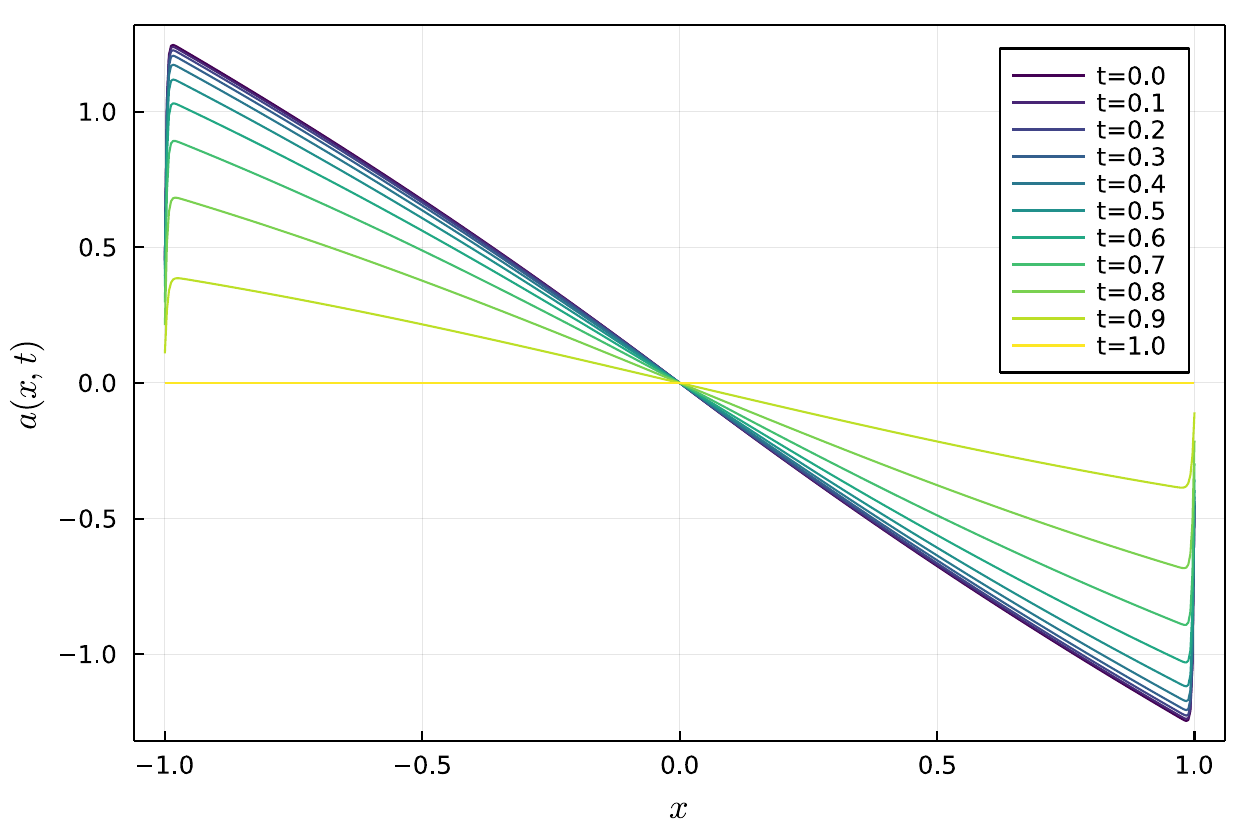}
      \caption{Input $a$.}
      \label{fig:test2-1u}
    \end{subfigure}
    \caption{Time evolution of the density $m$ and the input $a$ when $k=1000$ and \rev{$L_1=0$}.}
    \label{fig:test2-1}
  \end{figure}

  \begin{figure}[ht]
    \begin{subfigure}{0.49\textwidth}
      \includegraphics[width=\linewidth]{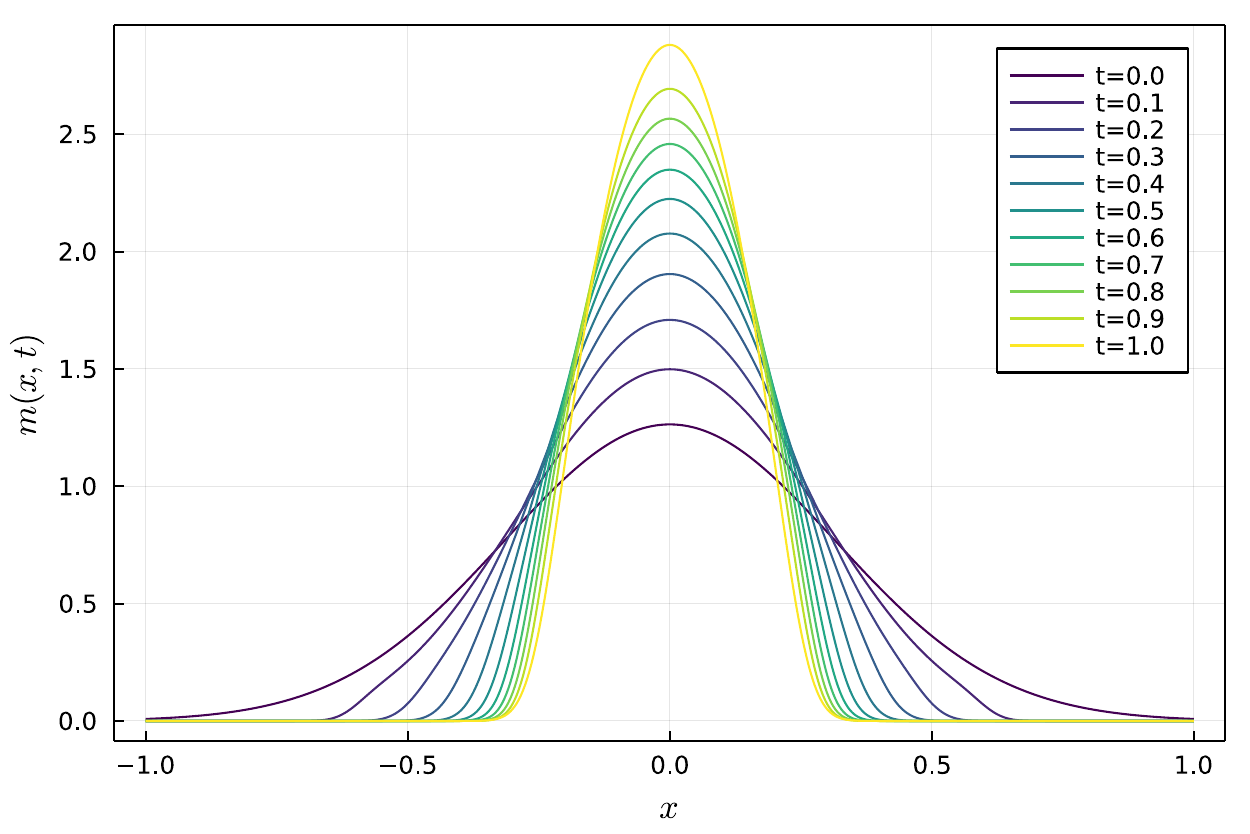}
      \caption{Density $m$.}
      \label{fig:test2-2m}
    \end{subfigure}
    \begin{subfigure}{0.49\textwidth}
      \includegraphics[width=\linewidth]{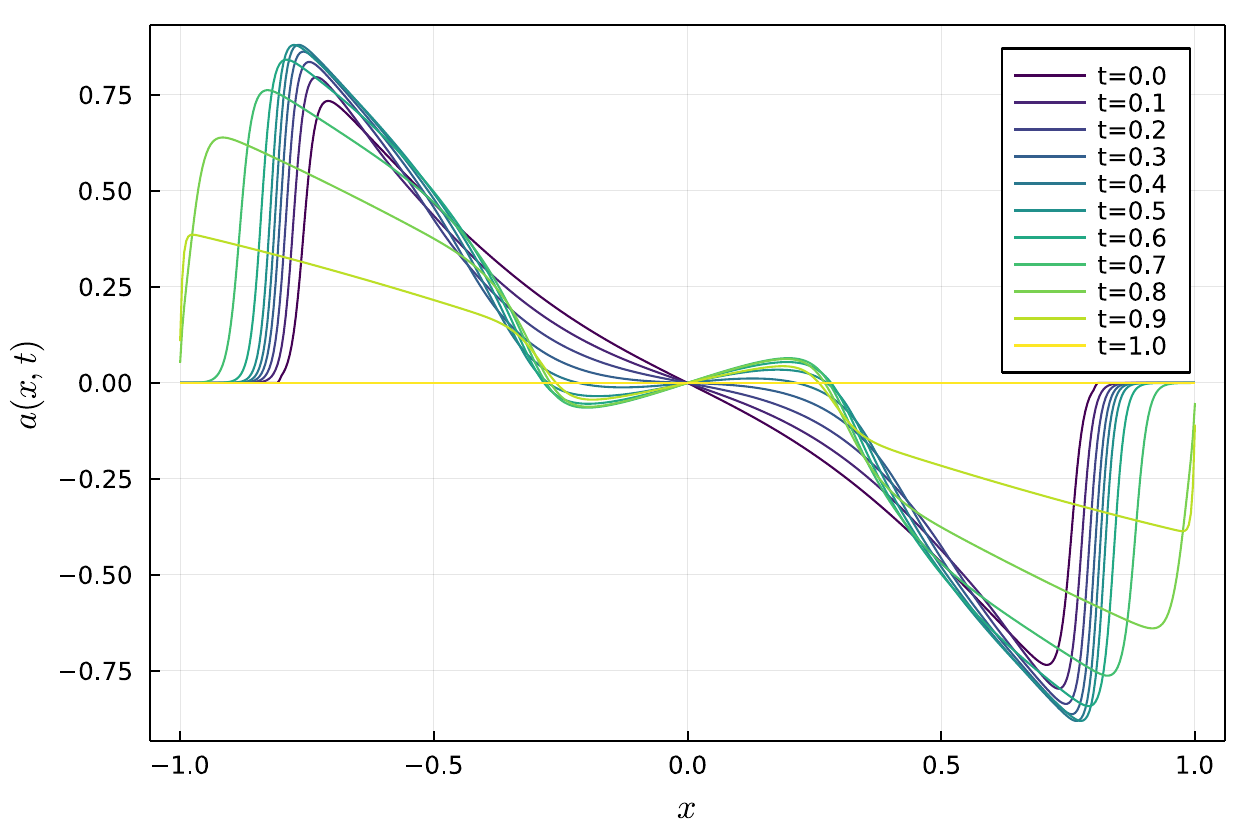}
      \caption{Input $a$.}
      \label{fig:test2-2u}
    \end{subfigure}
    \caption{Time evolution of the density $m$ and the input $a$ when $k=1000$ and \rev{$L_1=0.1$}.}
    \label{fig:test2-2}
  \end{figure}
  
  Figures~\ref{fig:test2-1}--\ref{fig:test2-2} show the time response of the density function $m^{(k)}$ and the input function $a^{(k)}$ at the end of the $k=1000$th iteration of fictitious play.
  Figure~\ref{fig:test2-1} shows the plot when \rev{$L_1=0$}.
  As a result of stabilizing the unstable system, the density is concentrated near $x=0$.
  To avoid this, we set \rev{$L_1=0.1$}; the result is shown in Fig.~\ref{fig:test2-2}.
  Generating an attractive input at the far origin stabilizes the system near the origin, whereas generating a repulsive input near the origin suppresses the peak of the density distribution.

  We here examine the complexity and convergence speed of the proposed scheme.
  First, we measure the CPU execution time required for the norm $\|M^{(k)}-M^{(k-1)}\|_\infty$ to be less than $5.0\times10^{-7}$.
  \begin{figure}[t]
    \centering
    \includegraphics[width=0.7\linewidth]{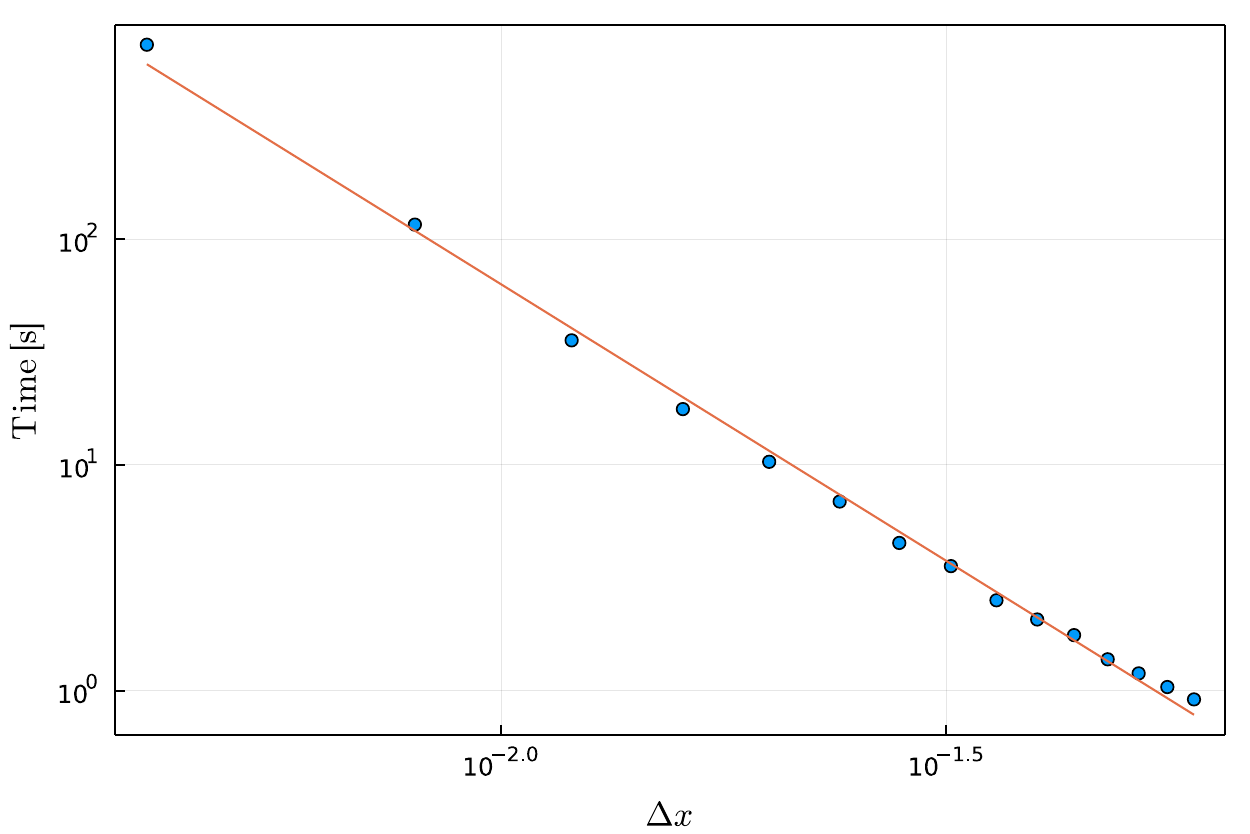}
    \caption{
      Execution time until $\|M^{(k)}-M^{(k-1)}\|_\infty<5.0\times10^{-7}$ is satisfied for various $\Delta x$.
    }
    \label{fig:sweep_time_1d}
  \end{figure}
  The results of calculating the time for various $\Delta x$ are plotted in Fig.~\ref{fig:sweep_time_1d}.
  The parameter $\Delta t$ is set so that the left-hand side of the CFL condition \eqref{eq:CFL} becomes $0.5$:
  \begin{align}\label{eq:numerics_delta_t_1d}
    % \Delta t = 0.5 \left( \frac{\sup_{x\in\Omega} |f(x)|}{\Delta x} + \frac{\sigma^2}{\Delta x^2} \right)^{-1}.
    \Delta t = 0.5 \left( \frac{\|f\|_\infty}{\Delta x} + \frac{\sigma^2}{\Delta x^2} \right)^{-1}.
  \end{align}
  Note that we have $\Delta t\sim \Delta x$ since $\sigma$ is small compared to $f$.
  In Fig.~\ref{fig:sweep_time_1d}, the execution time becomes larger as the parameter $\Delta x$ becomes smaller.
  The order is found to be approximately $O(\Delta x^{-2.45})$.

  \rev{We compare this computation time with the existing studies.
  Reference~[13], which proposes a similar iterative method to ours, evaluates the time to solve a one-dimensional control problem. 
  They conclude that the time taken for the solution to converge is $O((\Delta x\Delta t)^{-1})$. 
  This implies $O(\Delta x^{-2})$ since their method requires $\Delta x\sim \Delta t$.
  From these facts, we believe that the proposed method is comparable to the existing method for computation time.}

  Next, we clarify the convergence speed.
  As confirmed from the proof of Theorem \ref{thm:convergence-repeat-limit}, the error between the solution $m$ of the MFG and the solution $M^{(k)}$ of the scheme is bounded as follows:
  \begin{align}\label{eq:m_bound_abst}
    \|m-M^{(k)}\|_\infty \le K(k) \rev{\qty((\Delta x)^\eta + (\Delta t)^\frac{\eta}{2})} + L(k),
  \end{align}
  where the function $K(k)$ on the right-hand side is monotonically increasing with respect to $k$ and the function $L(k)$ satisfies $L(k) \searrow 0$ as $k\to\infty$.
  Note that the function $L(k)$ describes the convergence speed of fictitious play.  
  Since the true solution $m$ is unknown, we instead define the reference solution $\tilde M$, which is obtained by our scheme with sufficiently small $\Delta x$, $\Delta t$, and a sufficiently large $k$.
  We then calculate the relative error $\|\tilde M-M^{(k)}(\Delta x, \Delta t)\|_\infty$, where $M^{(k)}(\Delta x, \Delta t)$ denotes the solution with parameters $\Delta x, \Delta t, k$.
  We specifically define the reference solution $\tilde M$ with $\Delta x = 0.004$ and $k=1000$.

  \begin{figure}[t]
    \centering
    \includegraphics[width=0.72\linewidth]{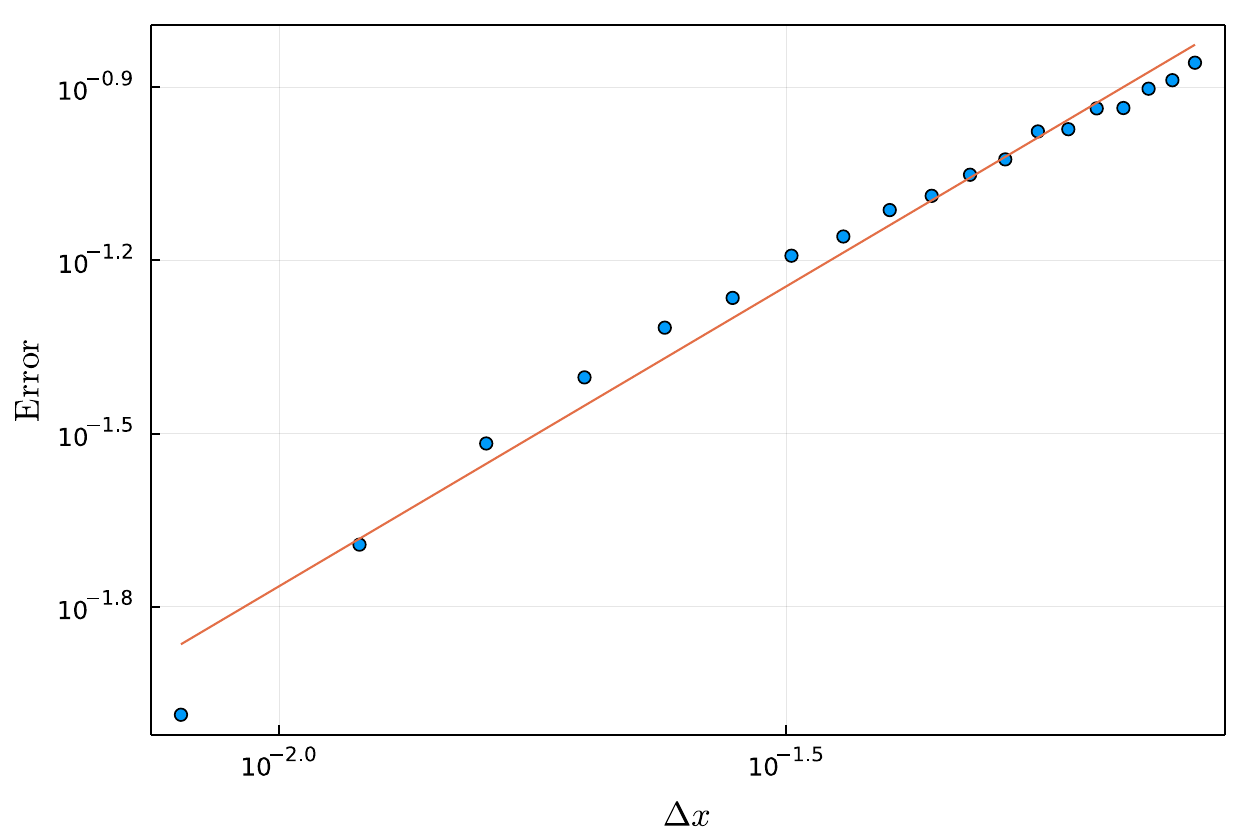}
    \caption{
      The relative error $\|\tilde M-M^{(k)}(\Delta x, \Delta t)\|_\infty$ between the solution $M^{(k)}(\Delta x, \Delta t)$ and the reference solution $\tilde M$ for various $\Delta x$.
    }
    \label{fig:sweep_dx_1d}
  \end{figure}
  First, we fix the parameter $k$ as $k=1000$ and plot the relative error for various $\Delta x$ in Fig.~\ref{fig:sweep_dx_1d}.
  The parameter $\Delta t$ is determined based on \eqref{eq:numerics_delta_t_1d}.
  A smaller parameter $\Delta x$ results in a smaller error.
  The order is found to be $O(\Delta x^{1.04})$, which supports the results in Theorem~\ref{thm:convergence-repeat}.
  \begin{figure}[t]
    \centering
    \includegraphics[width=0.7\linewidth]{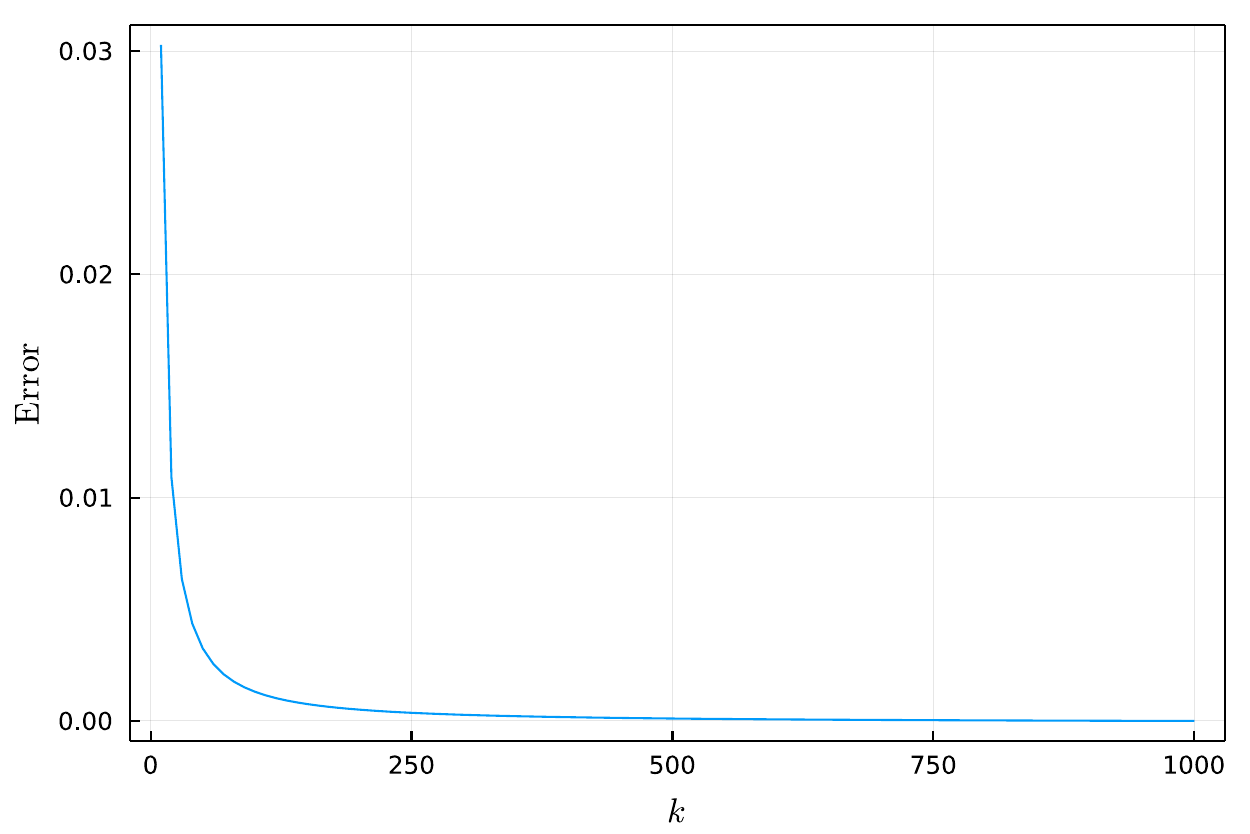}
    \caption{
      The relative error $\|\tilde M-M^{(k)}(\Delta x, \Delta t)\|_\infty$ between the solution $M^{(k)}(\Delta x, \Delta t)$ and the reference solution $\tilde M$ for various $k$.
    }
    \label{fig:sweep_k_1d}
  \end{figure}
  Next, we fix the parameter $\Delta x$ as $\Delta x=0.004$, determine $\Delta t$ based on \eqref{eq:numerics_delta_t_1d}, and calculate the relative error for various $k$. 
  The first term on the right-hand side of \eqref{eq:m_bound_abst} is small enough so that we expect to estimate the form of the function $L(k)$.
  The results are plotted in Fig.~\ref{fig:sweep_k_1d}.
  A smaller parameter $k$ results in an exponentially smaller error.
\end{example}
  
\begin{example}
  We consider the two-dimensional interval $\Omega=[-1,1]^2$ with a periodic boundary condition and consider the temporal parameter $T=1$.
  We use the following functions:
  \begin{align}
    f(x) &= \rev{Ax}, \\
    \begin{split}      
      g(x,m) &= (x-x_T)^\top Q (x-x_T) + S\exp(-(x-x_O)^\top\Sigma_O^{-1}(x-x_O))\\
      &\quad + \rev{L_1  \min\{m(x), L_2\}},
    \end{split}\\
    % \sigma &= 0.5, \ B = \mathrm{diag}(1,1),\ R = \mathrm{diag}(1,1),\\ 
    \sigma &= \mathrm{diag}(0.5,0.5), \ B = \mathrm{diag}(1,1),\ R = \mathrm{diag}(1,1),\\ 
    v_T(x) &= (x-x_T)^\top Q_T (x-x_T) + S_T\exp(-(x-x_O)^\top\Sigma_O^{-1}(x-x_O)),\\
    m_0(x) &\sim  \exp(-(x-x_I)^\top\Sigma_{I}^{-1}(x-x_I)),
  \end{align}
  where $x_T, x_I, x_O$ denote the destination, average initial position, and location of an obstacle, respectively.
  \rev{The matrix $A\in\bbR^{2\times 2}$ describes the dynamics of the system.}
  The parameters $Q$, $Q_T$, $\Sigma_I$, $\Sigma_O\in\bbR^{2\times 2}$ are positive-definitive matrices, and $S$, $S_T$, \rev{$L_1$, $L_2$}$\in\bbR$ are positive scalar parameters.
  We use \rev{$A = \mathrm{diag}(-0.05,-0.05)$}, $Q=\mathrm{diag}(1,1)$, $Q_T=\mathrm{diag}(0.8,0.8)$, $\Sigma_I=\mathrm{diag}(0.02, 0.02)$, $\Sigma_O=\mathrm{diag}(0.01, 0.01)$, $S=0.2$, $S_T=0.2$, \rev{$L_1\in\{0, 0.2\}$ and $L_2=10$} for the values for these parameters.
  We set $x_T=[0.5, 0.5]^\top$, $x_I=[-0.4, -0.4]^\top$, and $x_O=[-0.1, -0.1]^\top$.
  Reducing the above evaluation function means that each microscopic system tries to achieve the following four goals: 1) keep its own input as small as possible, 2) keep the position $x$ close to the destination $x_T$, 3) avoid the obstacle located at $x_O$ as much as possible, and 4) avoid congestion of the density. 
  \rev{Note that $A\preceq 0$ corresponds to a situation where the field is bowl-shaped and each player rolls down toward $x=0$ when there is no control input.}
  
  We define the discrete parameters as $\Delta x_1 = \Delta x_2 = 0.05$ and $\Delta t = 0.004$, so that the CFL condition 
  \begin{align}
    \Delta t \sum_{l\in\{1,2\}} \left( \frac{\|f_l\|_\infty }{\Delta x_l} + 2 \frac{ \nu_{ll} }{\Delta x_l^2}\right) = \rev{0.88} < 1,
  \end{align}
  is satisfied.

  \begin{figure}[t]
    \centering
    \includegraphics[width=0.9\linewidth]{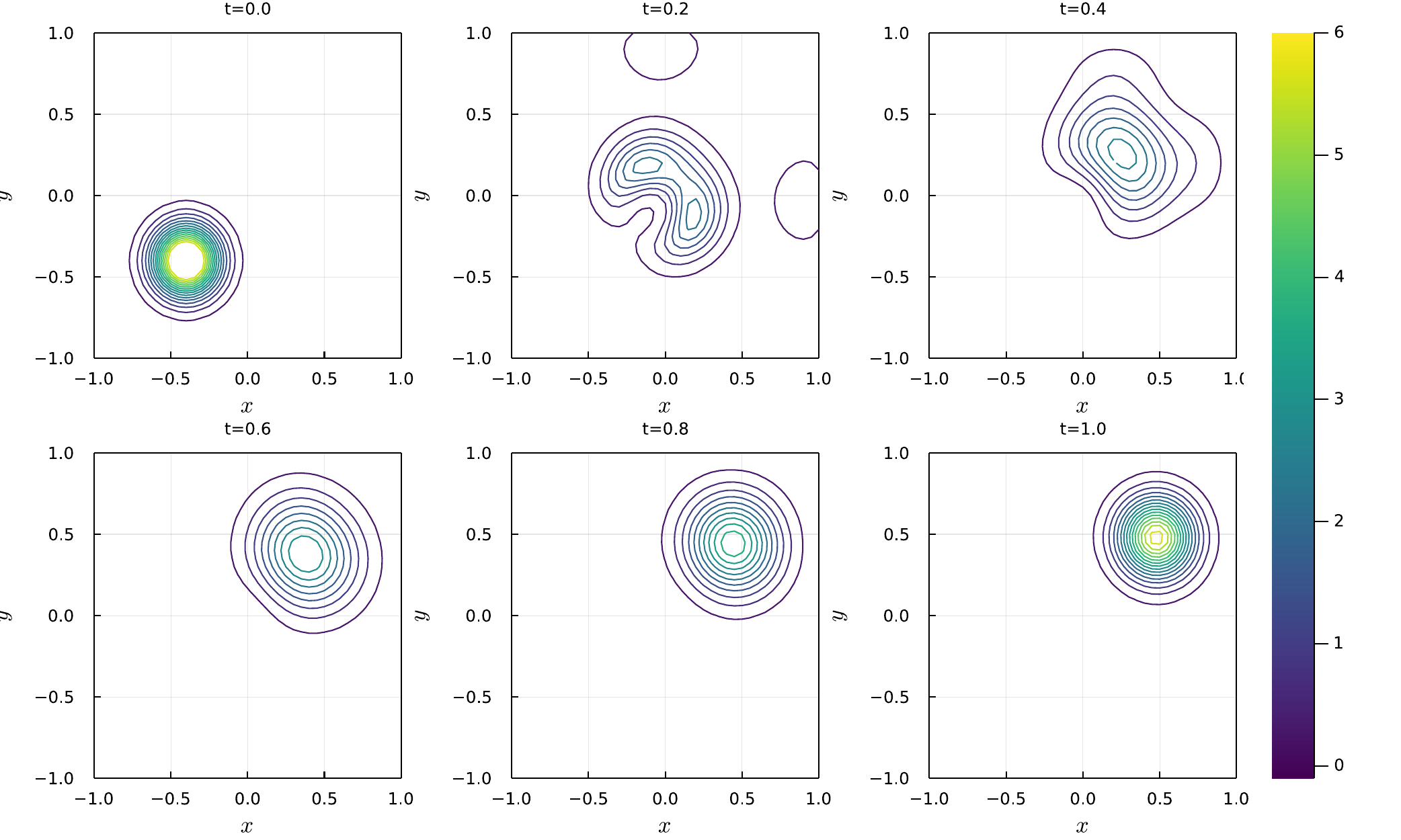}
    \caption{
    Trajectory of the density function $m^{(k)}$ when $k=500$ and \rev{$L_1=0$}. }
    \label{fig:time-evolution_0}
  \end{figure}
  
  \begin{figure}[t]
    \centering
    \includegraphics[width=0.9\linewidth]{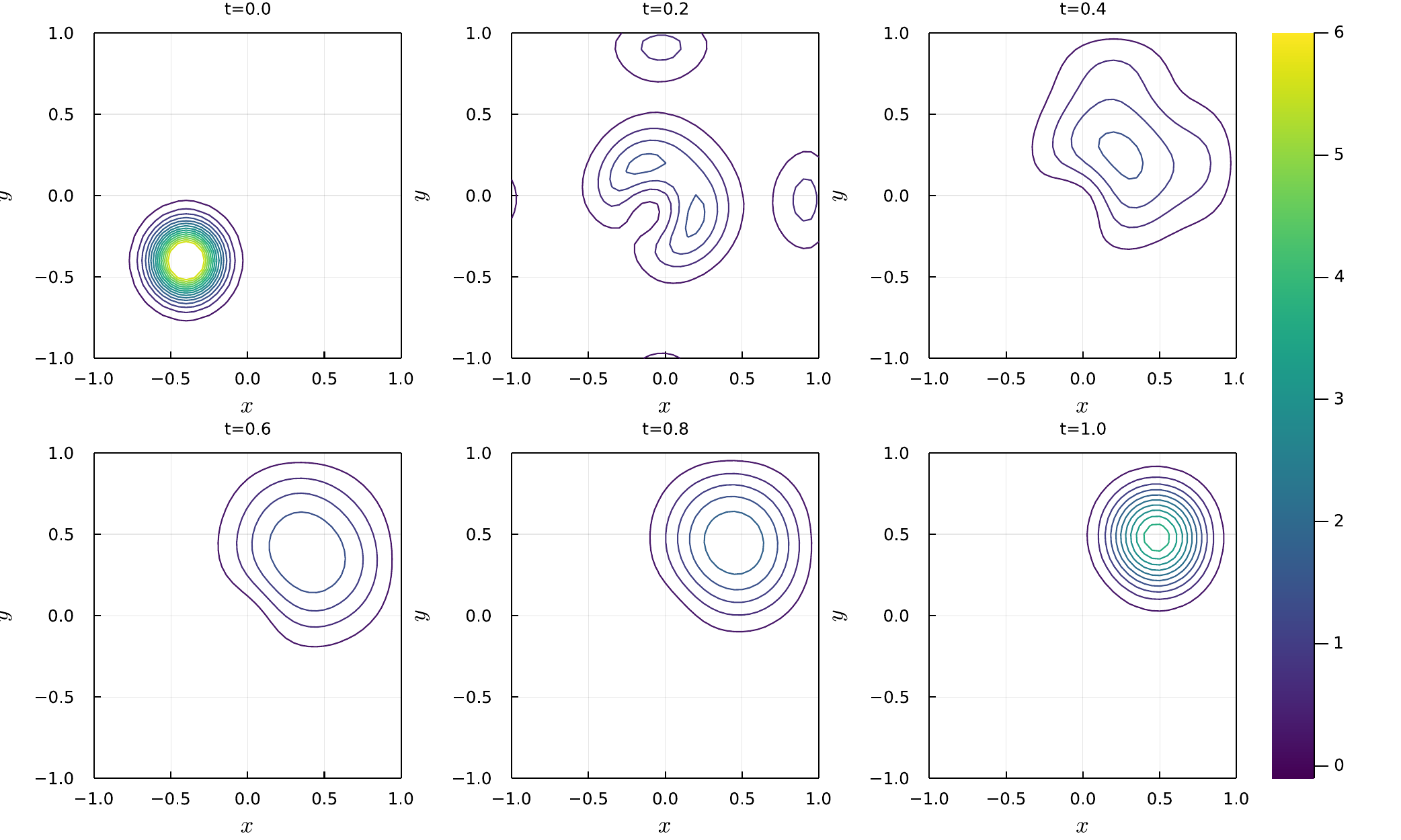}
    \caption{
    Trajectory of the density function $m^{(k)}$ when $k=500$ and \rev{$L_1=0.2$}. }
    \label{fig:time-evolution_1}
  \end{figure}

  Figures~\ref{fig:time-evolution_0} and \ref{fig:time-evolution_1} show the density profile after the $k=500th$ iteration of fictitious play.
  Figure~\ref{fig:time-evolution_0} shows the time response of the density function $m^{(k)}$ when \rev{$L_1=0$}.
  We see that each microscopic system moves toward the target coordinate $x=x_T$ while avoiding the area of $x=x_O$.
  Congestion occurs, however, especially near the end time $t=1$.
  On the other hand, Fig.~\ref{fig:time-evolution_1} shows the trajectory of the density when \rev{$L_1=0.2$}.
  Congestion is suppressed by setting the parameters to encourage the microscopic system to avoid the high-density area.

  We can now investigate the complexity and convergence speed of the proposed scheme.
  \begin{figure}[t]
    \centering
    \includegraphics[width=0.7\linewidth]{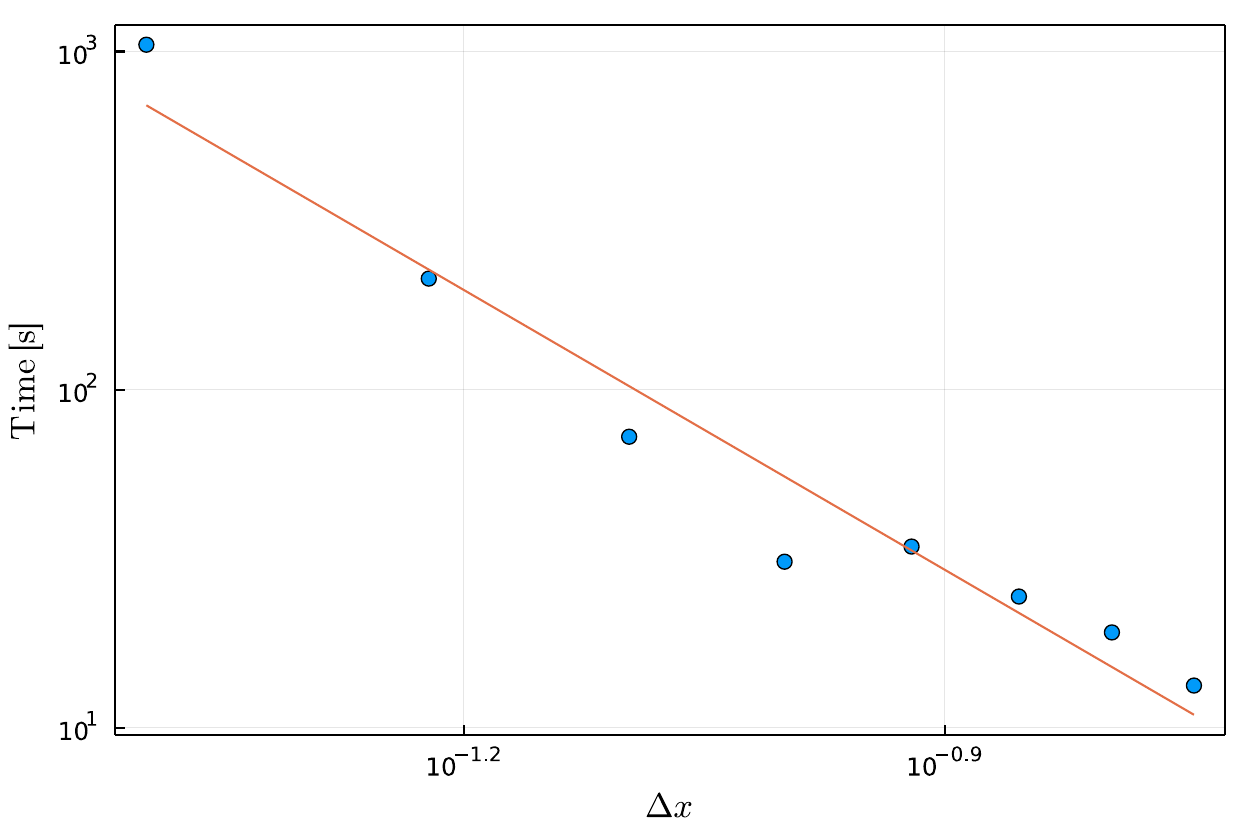}
    \caption{
      Execution time until $\|M^{(k)}-M^{(k-1)}\|_\infty<1.0\times10^{-4}$ is satisfied for various $\Delta x$.
    }
    \label{fig:sweep_time_2d}
  \end{figure}
  The CPU execution time required for the norm $\|M^{(k)}-M^{(k-1)}\|_\infty$ to be less than $1.0\times10^{-4}$ is plotted in Fig.~\ref{fig:sweep_time_2d}.
  Here, we use the same parameter $\Delta x_1=\Delta x_2\eqqcolon\Delta x$ and set $\Delta t$ so that the left-hand side of the CFL condition becomes $0.8$:
  \begin{align}\label{eq:numerics_delta_t_2d}
    \Delta t = 0.8\left\{ \sum_{l\in\{1,2\}} \left( \frac{\|f_l\|_\infty }{\Delta x} + 2 \frac{ \nu_{ll} }{\Delta x^2}\right)\right\}^{-1}.
  \end{align}
  Note that we have $\Delta t\sim (\Delta x)^2$, since $\sigma$ is large compared to $f$.
  In Fig.~\ref{fig:sweep_time_2d}, the execution time becomes larger as the parameter $\Delta x$ becomes smaller.
  The order is found to be \rev{$O(\Delta x^{-2.76})$}.
 
  Next, we check the convergence speed.
  We define the reference solution $\tilde M$ with the parameters $\Delta x = 0.04$ and $k=500$.
  We then calculate the relative error $\|\tilde M-M^{(k)}(\Delta x, \Delta t)\|_\infty$, where $M^{(k)}(\Delta x, \Delta t)$ denotes the solution with parameters $\Delta x, \Delta t, k$.
  \begin{figure}[t]
    \centering
    \includegraphics[width=0.72\linewidth]{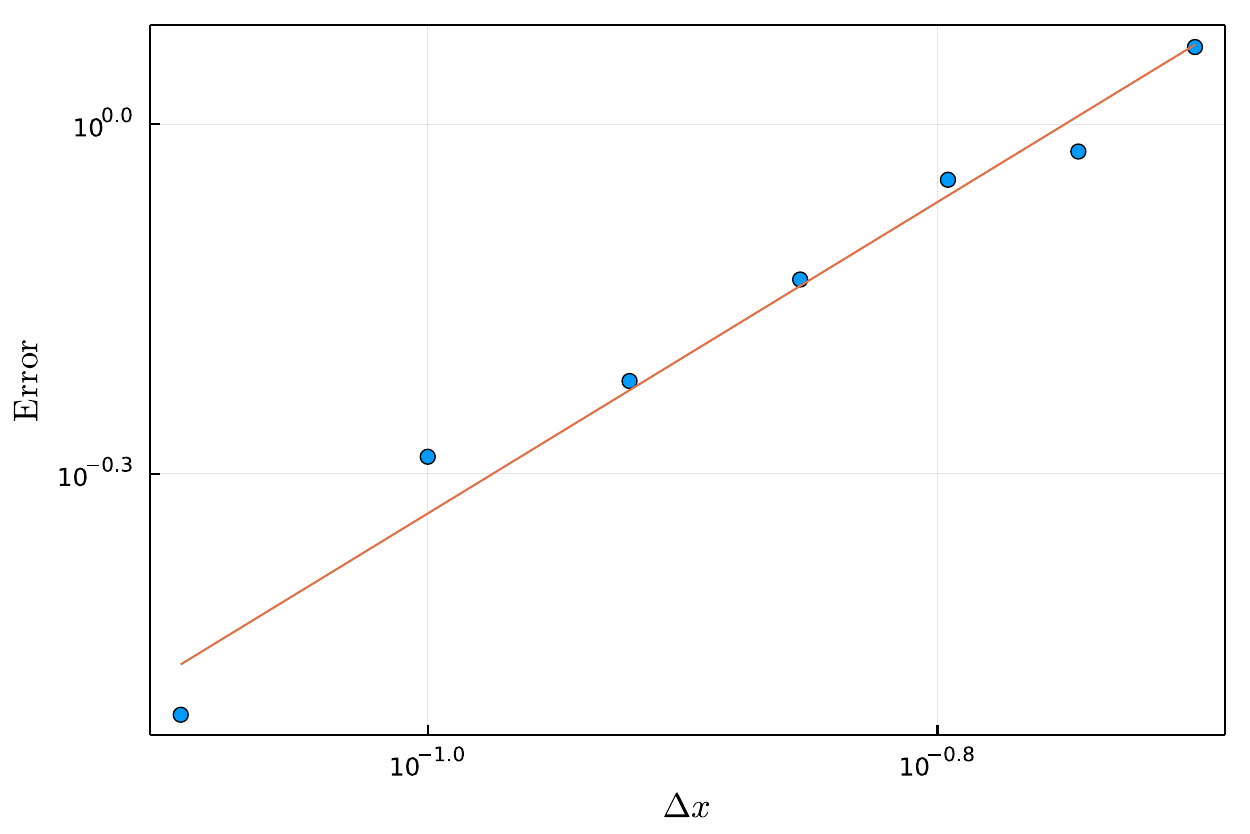}
    \caption{
      The relative error $\|\tilde M-M^{(k)}(\Delta x, \Delta t)\|_\infty$ between the solution $M^{(k)}(\Delta x, \Delta t)$ and the reference solution $\tilde M$ for various $\Delta x$.
    }
    \label{fig:sweep_dx_2d}
  \end{figure}
  First, we fix the parameter $k$ as $k=500$ and plot the relative error for various $\Delta x$ in Fig.~\ref{fig:sweep_dx_2d}.
  The parameter $\Delta t$ is determined based on \eqref{eq:numerics_delta_t_2d}.
  As the parameter $\Delta x$ becomes smaller, the error becomes smaller.
  The order is found to be approximately \rev{$O(\Delta x^{1.33})$}.
  \begin{figure}[t]
    \centering
    \includegraphics[width=0.7\linewidth]{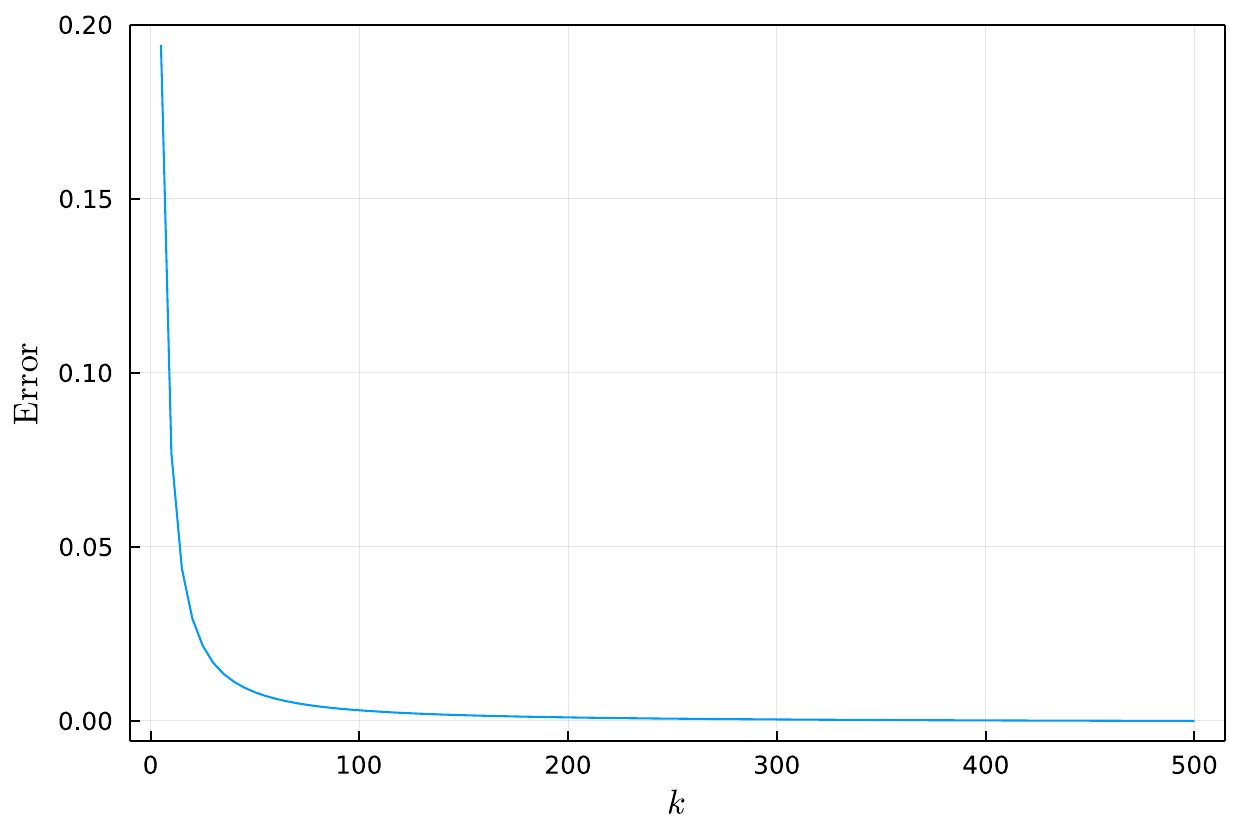}
    \caption{
      The relative error $\|\tilde M-M^{(k)}(\Delta x, \Delta t)\|_\infty$ between the solution $M^{(k)}(\Delta x, \Delta t)$ and the reference solution $\tilde M$ for various $k$.
    }
    \label{fig:sweep_k_2d}
  \end{figure}
  Next, we fix the parameter $\Delta x$ as $\Delta x=0.04$, determine $\Delta t$ based on \eqref{eq:numerics_delta_t_2d}, and calculate the relative error for various $k$. 
  The results are plotted in Fig.~\ref{fig:sweep_k_2d}.
  We find that the error is exponentially smaller as the parameter $k$ is larger, which is the same result as in the previous example.
\end{example}

\section{Conclusion}\label{sec:conclusion}
  
In this paper, we derived linearly solvable mean field games (LS-MFGs) by using the Cole-Hopf transformation for mean field games.
We proposed a numerical scheme for the LS-MFG in which the forward equation and the backward equation are solved alternately with fictitious play iterations.
We also proved that the obtained scheme has suitable convergence properties.
Specifically, we first showed that each of the forward and backward equations satisfies the discrete maximum principle (in Propositions \ref{prop:bound_Psi}, \ref{prop:bound_tildePsi}) and converges (in Propositions~\ref{thm:convergence}, \ref{thm:convergence-conservative}), and then showed convergence when the two equations are solved iteratively (in Theorems~\ref{thm:convergence-repeat}, \ref{thm:convergence-repeat-limit}).

Proof for the case of more than two dimensions is expected to be easy but is left as an issue for the future.
\rev{In addition, convergence analysis for the case of weakening the regularity assumption for the coupling term should be listed as an open problem.}
\rerev{Furthermore, convergence analysis when generalizing fictitious play remains a topic of important future work.}

\section*{Acknowledgement}
The authors thank the anonymous reviewers for their valuable comments and suggestions to improve the quality of the paper.

\section*{Appendix}
\renewcommand{\theequation}{A.\arabic{equation}}

\subsection*{Proof of Proposition \ref{prop:CH-backward-bound}}

The boundedness is obvious because of the existence of solutions to the HJB equation \eqref{eq:LQHJB} and the definition of the Cole-Hopf transformation \eqref{eq:CH-backward}.
We show that there exists $\gamma_{\min}>0$ such that $\psi\ge \gamma_{\min}\ \forall (x,t)$.
Setting $\phi(x,t) = \exp(Ut)\psi(x,t)$ with $U\ge0$, we transform \eqref{eq:CH-backward} and \eqref{eq:CH-backward-0} into
\begin{align}
    \partial_t \phi(x ,t) &= -  f(x)^\top\partial_x \phi(x,t) - \Tr \{\nu \partial_{xx}\phi(x,t)\} + \qty{\frac{1}{\lambda}\rev{g(x,m(\cdot,t))} + U}\phi(x,t)\qquad (x,t)\in\Omega\times [0,T],\\
    \phi(x,T) &= \exp(UT-\frac{v_T(x)}{\lambda}) \qquad x\in\Omega.
\end{align}
Let $\underline\phi(t)$ be a solution of the following differential equation:
\begin{align}
  \drv{}{t}\underline\phi(t) &= \qty(U + \frac{1}{\lambda}\|g\|_{\infty} ) \underline{\phi}(t),\\
  \underline\phi(T) &= \exp(UT-\frac{\|v_T\|_{\infty}}{\lambda}).
\end{align}
This variable $\underline\phi(t)$ satisfies
\begin{align}
  \underline\phi(t) &\ge \exp(-\frac{1}{\lambda} (\|v_T\|_{\infty} + \|g\|_{\infty}T) )\eqqcolon\bar\gamma_{\min}.
\end{align}
We define $K(t)\coloneqq \frac12 \int_\Omega (\underline\phi(t) - \phi(x,t) )_+^2\rmd x$.
\rerev{
  Then, 
  \begin{align}
    \frac{\rmd}{\rmd t} K(t) = K_1 + K_2 + K_3 + K_4 + K_5,
  \end{align}  
  where
  \begin{align}
    \begin{aligned}
    K_1 & =\int \frac{1}{\lambda}(\|g\|_\infty-g(x,m))(\phi(x,t)-\underline{\phi}(t))(\underline{\phi}(t)-\phi(x,t))_{+} \rmd x, \\
    K_2 & =\int \frac{1}{\lambda}(\|g\|_\infty-g(x,m)) \underline{\phi}(t)(\underline{\phi}(t)-\phi(x,t))_{+} \rmd x, \\
    K_3 & =\int\left(-f(x)^{\top} \partial_x(\underline{\phi}(t)-\phi(x,t))+U(\underline{\phi}(t)-\phi(x,t))\right)(\underline{\phi}(t)-\phi(x,t))_{+} \rmd x, \\
    K_4 & =\int \operatorname{Tr}\left(\nu \partial_{x x} \phi(x,t)\right)(\underline{\phi}(t)-\phi(x,t))_{+} \rmd x, \\
    K_5 &= \int \frac{1}{\lambda}\|g\|_{\infty}(\underline\phi(t)-\phi(x, t))(\underline\phi(t)-\phi(x, t))_+ \rmd x.
    \end{aligned}
  \end{align}
  Note first that $K_2$ and $K_5$ are non-negative.
  Using partial integration, it is shown that $K_4$ is also non-negative.
  Applying next Young's inequality and choosing
  \begin{align}
    U=C \|f\|_\infty+\frac{2}{\lambda}\|g\|_\infty,
  \end{align}
  for sufficiently large $C>0$, we obtain that $K_1+K_3$ is also non-negative. 
  Hence, we obtain $\drv{}{t} K(t)\ge 0$.
  This and $K(T)=0$ shows $K(t)\equiv0$, which means $\phi\ge \bar\gamma_{\min}>0$.
  Taking $\gamma_{\min} \coloneqq \bar\gamma_{\min}/\exp(UT)>0$, we have $\psi\ge \gamma_{\min}$, which completes the proof.
}

\subsection*{Proof of Proposition \ref{prop:d1-sup}}

\rerev{
  Fix a point $x_0\in\Omega$ and take $h$ as a 1-Lipschitz function. 
  Then
  \begin{align}
    \begin{aligned}
    \int h(x) (m(x) - m'(x))\rmd x&=\int\left(h(x)-h\left(x_0\right)\right)(m(x) - m'(x))\rmd x\\
    & \leq \operatorname{diam}(\Omega)\|m-m'\|_{L^1} \\
    & \leq \operatorname{diam}(\Omega) \operatorname{vol}(\Omega)\|m-m'\|_{\infty} .
    \end{aligned}
  \end{align}
  Since $h$ is arbitrary, the result follows immediately with
  \begin{align}
    C_{\Omega}\coloneqq \operatorname{diam}(\Omega) \operatorname{vol}(\Omega) .
  \end{align}
}

\end{document}

%% file: macro.tex
\usepackage{xcolor}
\newcommand\rev[1]{#1}
\newcommand\rerev[1]{#1}

\newcommand{\sgn}{\text{sgn}}

\newcommand{\drv}[2]{\frac{\rmd#1}{\rmd#2}}

\newcommand{\calC}{{\mathcal C}}

\newcommand{\calP}{{\mathcal P}}

\newcommand{\bbE}{{\mathbb E}}

\newcommand{\bbN}{{\mathbb N}}

\newcommand{\bbR}{{\mathbb R}}

\newcommand{\bbZ}{{\mathbb Z}}

\newcommand{\bfone}{{\mathbf 1}}

\newcommand{\bfd}{{\mathbf d}}

\newcommand{\bfi}{{\mathbf i}}

\newcommand{\rmd}{{\mathrm d}}

\newcommand{\tile}{{\tilde e}}

\newcommand{\tilr}{{\tilde r}}

\newcommand{\tilC}{{\tilde C}}
\newcommand{\tilD}{{\tilde D}}
\newcommand{\tilE}{{\tilde E}}

\newcommand{\tilK}{{\tilde K}}

\newcommand{\tilR}{{\tilde R}}

\makeatletter  %@を文字として認識できるようにする．
\newcommand{\repeatable}[2]{%
    \global\@namedef{repeatable@#2}{#1}#1 \label{#2}  %変数 repeatable@#2 に #1 を代入．本文に #1 \label{#2} を表示．
}
\newcommand{\repeatref}[1]{%
    \@ifundefined{repeatable@#1}{NOT FOUND}{\footnote[0]{\eqref{#1}$ \displaystyle{ \@nameuse{repeatable@#1} } $}}%
    ~\eqref{#1}} % repeatable@#1 が未定義ならNOT FOUNDを表示．定義済みなら脚注に式番号と数式を表示して本文に式番号を表示．
\makeatother %@を文字として認識しないようにする．